\newcommand{\R}{\mathbb{R}}			
\newcommand{\N}{\mathbb{N}}			
\newcommand{\Pb}{\mathscr{P}}		
\newcommand{\Pbac}{\mathscr{P}^{ac}}
\newcommand{\Xc}{\mathcal{X}}		

\newcommand{\T}{{[0,T]}}

\newcommand{\I}{[a,b]}				
	
\newcommand{\J}{{[0,1]}}			
\newcommand{\Jo}{{(0,1)}}
\newcommand{\TJ}{{[0,T]\times[0,1]}}
\newcommand{\TJoo}{{(0,T)\times(0,1)}}

\newcommand{\TJoc}{{(0,T)\times[0,1]}}

\newcommand{\Var}{\mathrm{Var}}

\renewcommand{\c}{{\sf c}}
\newcommand{\dc}{\c^\prime}
\newcommand{\ddc}{\c^\pprime}
\newcommand{\ddcb}{\underline{\ddc}} 		
\newcommand{\cs}{\c^*}
\newcommand{\dcs}{(\c^*)^\prime}
\newcommand{\ddcs}{(\c^*)^\pprime}
\newcommand{\ct}{{\tilde \c}}

\newcommand{\ls}{\gamma}
 
\newcommand{\hu}{\mathsf{h}_m}				

\newcommand{\hX}{h_m}						
\newcommand{\dhX}{h_m^\prime}				
\newcommand{\ddhX}{h_m^\pprime}				
\newcommand{\sh}{\mathsf{h}}				

\newcommand{\vu}{\mathsf{v}}				
\newcommand{\dvu}{\mathsf{v}^\prime}

\newcommand{\vX}{v}							
\newcommand{\dvX}{v^\prime}
\newcommand{\ddvX}{v^\pprime}
\newcommand{\ddvXa}{\overline{\ddvX}}		

\newcommand{\Wuc}{\mathsf{W}_{\c}}
\newcommand{\Wuct}{\mathsf{W}_{\c,\tau}}

\newcommand{\WXc}{\mathbf{W}_{\c}}
\newcommand{\WXct}{\mathbf{W}_{\c,\tau}}

\newcommand{\WXdctk}{W_{\c,\tau_k}}

\newcommand{\Hu}{\mathsf{H}_m}

\newcommand{\HX}{\mathbf{H}_m}

\newcommand{\HXd}{H_m}
\newcommand{\HXdb}{\underline{H_m}}				
\newcommand{\VXd}{V}
\newcommand{\VXdb}{\underline{V}}				
\newcommand{\UXd}{U}

\newcommand{\F}{\mathfrak{F}}		
\newcommand{\g}{\mathbf{g}}

\newcommand{\dk}{\frac{1}{k}}

\newcommand{\bx}{\mathbf{x}}
\newcommand{\bxn}[1]{\mathbf{x}^{(#1)}}

\newcommand{\xn}[1]{x^{(#1)}}
\newcommand{\dxub}{\overline{\delta\bxn{0}}}
\newcommand{\dxlb}{\underline{\delta\bxn{0}}}
\newcommand{\ddxub}{\overline{\delta^2\bxn{0}}}
\newcommand{\ddxlb}{\underline{\delta^2\bxn{0}}}

\newcommand{\Xn}[1]{X^{(#1)}}
\newcommand{\dexi}{\delta_\xi}
\newcommand{\detau}{\delta_{\tau_k}}

\newcommand{\Aink}{a_{i,k}^{(n)}}
\newcommand{\bA}{\textbf{A}}			

\newcommand{\Ak}{A_k}
\newcommand{\olAk}{\overline{A}_k}			 
\newcommand{\ulAk}{\underline{A}_k}

\newcommand{\As}{A_*} 						

\newcommand{\CdelH}{\mathfrak{V}}

\newcommand{\CdelHs}{\mathfrak{V}_*}

\newcommand{\IndTX}{\mathcal{I}^{(n)}_{i,k}}

\newcommand{\as}{\text{(A.1)}}
\newcommand{\aspa}{\text{(A.2)}}
\newcommand{\aspb}{\text{(A.3)}}
\newcommand{\asfa}{\text{(Af.1)}}
\newcommand{\asfb}{\text{(Af.2)}}
\newcommand{\asfc}{\text{(Af.3)}}
\newcommand{\asfd}{\text{(Af.4)}}
\newcommand{\asv}{\text{(A.V.)}}

\newcommand{\de}{\;\mathrm{d}}						
\newcommand{\Lc}{\mathcal{L}} 						
\newcommand{\glqq}{``}								
\newcommand{\grqq}{''}
\newcommand{\norm}[1]{\left\| #1 \right\|}			
\newcommand{\abs}[1]{\left| #1 \right|}
\newcommand{\pprime}{{\prime\prime}}				

\newcommand{\ind}[1]{\mathbf{1}_{#1}}				
\newcommand{\eexp}[1]{\mathrm{e}^{#1}}				
\newcommand{\floor}[1]{\left\lfloor #1 \right\rfloor}	

\newcommand{\argmin}{\rm \arg min\,}

\documentclass[reqno]{amsproc}
\usepackage[foot]{amsaddr}

\usepackage{amssymb}

\usepackage{graphicx}

\usepackage{amsmath}
\usepackage{amsbsy}
\usepackage{amsopn}
\usepackage{dsfont}
\usepackage{multicol}
\usepackage{enumerate}
\usepackage{caption}
\usepackage{pgfplots}
\usepackage{listings}

\definecolor{colKeys}{rgb}{0,0,1} 
\definecolor{colIdentifier}{rgb}{0,0,0} 
\definecolor{colComments}{rgb}{0,1,0.3} 
\definecolor{colString}{rgb}{0,0.5,0} 

\definecolor{dkgreen}{rgb}{0,0.6,0} 
\definecolor{gray}{rgb}{0.5,0.5,0.5} 
\definecolor{lightgray}{rgb}{0.9,0.9,0.9} 

\usepackage{empheq}

\usepackage{mathrsfs}

\usepackage{geometry}

\usepackage{subfigure}

\newtheorem{theorem}{Theorem}
\newtheorem{lemma}[theorem]{Lemma}
\newtheorem{corollary}[theorem]{Corollary}
\newtheorem{proposition}[theorem]{Proposition}
\newtheorem*{algorithm}{Algorithm}

\theoremstyle{definition}
\newtheorem{definition}[theorem]{Definition}

\theoremstyle{remark}
\newtheorem{remark}[theorem]{Remark}
\newtheorem{remarks}[theorem]{Remarks}

\numberwithin{equation}{section}

\begin{document}
\raggedbottom
\title[Discretization of $p$-Wasserstein equations]{A convergent Lagrangian discretization for $p$-Wasserstein and flux-limited diffusion equations}

\author{Benjamin S\"ollner}
\address{Zentrum Mathematik, TU M\"unchen, Boltzmannstr. 3, D-95748 Garching, Germany}
\email{b.soellner@ma.tum.de}
\thanks{}

\author{Oliver Junge}
\email{oj@ma.tum.de}
\thanks{}


\begin{abstract}
We study a Lagrangian numerical scheme for solving a nonlinear drift diffusion equations of the form $\partial_t u = \partial_x(u \cdot\dcs[\partial_x\sh^\prime(u)+\dvu])$, like Fokker-Plank and $q$-Laplace equations, on an interval. This scheme will consist of a spatio-temporal discretization founded on the formulation of the equation in terms of inverse distribution functions. It is based on the gradient flow structure of the equation with respect to optimal transport distances for a family of costs that are in some sense $p$-Wasserstein like. Additionally we will show that, under a regularity assumption on the initial data, this also includes a family of discontinuous, flux-limiting cost inducing equations like Rosenau's relativistic heat equation. We show that this discretization inherits various properties from the continuous flow, like entropy monotonicity, mass preservation, a minimum/maximum principle and flux-limitation in the case of the corresponding cost. Convergence in the limit of vanishing mesh size will be proven as the main result. Finally we will present numerical experiments including a numerical convergence analysis. 
\end{abstract}

\maketitle

\section{Introduction}
In this paper, we want to study a spatio-temporal discretization of a family of parabolic equations,
\begin{equation}\label{PDE_u}
\partial_t u = \partial_x(u \cdot\dcs[\partial_x\sh^\prime(u)+\dvu])
\end{equation} 
for a probability density $u$ on a bounded interval $I=\I$ where $u$ is a probability density. Equations like both, the linear and non-linear Fokker-Plank equation, the $q$-Laplace equation and Rosenau's relativistic heat equation are included. 

This equation will be equipped with no-flux boundary conditions 
\begin{equation}\label{BddCond_u}
\partial_x u(t,a)=\partial_x u(t,b)=0 \quad \text{ for }t\geq 0
\end{equation}
and initial conditions
\begin{equation}\label{IniCond_u}
u(0,x)=u_0(x) \quad \text{ for }x\in \I
\end{equation}
where $M\geq u_0\geq 1/M$ for some $M>0$ and $\norm{u_0}_{L^1(\I)}=1$.

We will consider external potentials $\vu\in C^2(\I)$ which will be convex and internal energy potentials $\hu:[0,\infty)\rightarrow \R$ which will either be the Boltzmann entropy $\mathsf{h}_1(s)=s\log(s)+s$ (where we take $0\log(0)=0$) or the R\'{e}nyi entropy $\hu(s)=\frac{1}{m-1}s^m$.  

$\cs$ will denotes the Legendre-transform of $\c$ i.e. $\cs(s):=\sup_{r\in \R} (sr-\c(r))$. 
We consider a family of functions for $\c$, that can be called \glqq p-Wasserstein-like\grqq\ cost functions. A cost function is called a $p$-Wasserstein cost function (or Monge-Kantorovich cost of order $p$) if it has the form $s\mapsto c(s)=\frac{1}{p}\abs{s}^p$ (c.f. \cite{VillaniTopics}).

Let $p\in (1,\infty)$. To be a \glqq p-Wasserstein-like\grqq\ cost function, $\c$ has to have the following properties.
\begin{enumerate}
\item[\as] $\c$ is non-negative, strictly convex, continuous and even with $\c(0)=0$.
\item[\aspa] $\c\in C^1(\R)\cap C^2(\R\setminus\{0\})$.
\item[\aspb] There are constants $\alpha,\beta >0$ such that 
\begin{align*}
\alpha \abs{s}^{p} \leq s \dc(s) \leq \beta\abs{s}^{p} 
\end{align*}
holds for every $ s\in \R$.
\end{enumerate}
One family of examples would be the $p$-Wasserstein cost mentioned above $\c(s)=\frac{1}{p}\abs{s}^p$ for $p\in (1,\infty)$, which induces the $p$-Wasserstein distance and lends its name to this family of cost functions. 

We present two PDEs arising from certain choices of parameters.
\begin{itemize}
\item Let $q\in \N$ with $q>1$. Then choose $p\in (1,2)$ such that $\frac{2-p}{p-1}=q$ and pick $m=3-p>1$. For any constant external potential, the equation becomes the $q$-Laplace equation
\begin{align*}
\partial_t u= \partial_x(\abs{\partial_x u}^{q}\cdot \partial_x u)\;.
\end{align*}
The dynamic of such a $q$-Laplace equation can be seen in Fig.~\ref{fig:evol_rel_q-Laplace}. 
\item Let $p=2$. Then the equation has the form of a Fokker-Plank equation
\begin{align*}
\partial_t u = \frac{1}{m} \partial_x^2 u^m + \partial_x(u\dvu) \;.
\end{align*}
\end{itemize}

If we chose $\vu=\mathrm{const.}$, then the above properties of $\c$ will suffice to show the claims of this paper. If $\vu$ is not a constant, however, we require $\c$ to be $2$-Wasserstein like. For example, if $p\geq 2$, $\ddc(0)=0$ is implied which leads to $\dcs(r)=(\dc)^{-1}(r)$ not being Lipschitz any more at $r=0$. Therefore we cannot guarantee well posedness of our equation. 

Assuming some regularity of the initial data, another interesting family of cost functions is included in our calculations. Let $\partial_x u_0(x)$, the weak spatial derivative of our initial data, be bounded. We will consider a family of cost functions, called \glqq flux-limiting\grqq, satisfying \as\ and the following properties: 
\begin{itemize}
\item[\asfa] $\c^{-1}([0,\infty))=[-\ls,\ls]$ for some $\ls>0$, so it has to have a symmetric compact proper domain.
\item[\asfb] $\c\in C^2((-\ls,\ls))$. 
\item[\asfc] The limit $ \dc(s)\rightarrow \pm \infty$ for $s\rightarrow \pm \ls$ and  $s\in(-\ls,\ls)$ holds.
\item[\asfd] There is $\alpha,\beta >0$ such that 
\begin{align*}
\alpha \abs{s}^{2} \leq s \dc(s) \leq \beta \abs{s}^2
\end{align*}
holds.
\end{itemize}


An exemplary family of flux-limiting costs would be 
\begin{equation}\label{eq:c_rel}
\c:\R\rightarrow [0,\infty]\quad \text{ with } \quad \c(s)=\begin{cases}
\ls\left(1-\sqrt{1-\abs{\tfrac{s}{\ls}}^2}\right) &\text{ for } \abs{s}\leq \ls\\
+\infty &\text{ elsewhere,}
\end{cases}
\end{equation}
where $\ls>0$ (c.f. \cite{McCannPuel}). 

In these cases and with $m=1$ our equation becomes Roseanu's tempered diffusion equation introduced in \cite{rosenau1992tempered}:
\begin{align*}
\partial_t u = \ls \partial_x\left( u \cdot \frac{\partial_x u}{\sqrt{ u^2 +(\partial_x u)^2}} \right)\,.
\end{align*}

\subsection{Gradient flow structure} The equation \eqref{PDE_u} can be written as a transport equation 
\begin{align}\label{PDE_transport}
\partial_t u + \partial_x (u\cdot \mathcal{V}[u])=0
\end{align}
where the velocity $\mathcal{V}$ itself depends on the solution 
\begin{align}\label{PDE_transport_vel}
\mathcal{V}[u]=-\dcs[\partial_x\sh^\prime(u)+\dvu]\;. 
\end{align}
This form of \eqref{PDE_u} together with the no-flux boundary conditions implies conservation of mass, i.e. 
$$\int_I u(t,x)\de x=\int_I u_0(x)\de x \;. $$ 
 
To arrive at the in-time discretization of our problem we will make use of the following fact (c.f. \cite{Agueh}). Solutions to our equation form a gradient flow in the energy landscape of 
\begin{equation}\label{def:u_energy_landscape}
\Hu(\rho)=\begin{cases}
\displaystyle \int_I \hu(u(x))\de x+\int_I \vu(x)u(x)\de x &\text{ if } \rho = u \Lc \\
+\infty &\text{ if } \rho \in \Pb(\I)\setminus\Pbac(\I)
\end{cases} 
\end{equation}
with respect to the metric induced by the optimal transport distance with cost $\c$. Note that we will write $\Lc$ for the Lebesgue measure. 

\subsection{Inverse Distribution Functions}
We will carry out our analysis not on the densities $u$ but much rather on the inverse distribution functions $X$ of $u$. 

Recall that for a probability measure $\rho\in \Pb(I)$ with density $u$ such that $u+\frac{1}{u}\in L^\infty(I)$, the (cummulative) distribution function is given as $U_u(x):=\int_a^x u(\zeta) \de \zeta$ and the corresponding inverse distribution function is given by $X_u=U_u^{-1}$. Furthermore note that in this case, $X$ is a.e. differentiable with $\partial_\xi X= \frac{1}{u\circ X}$. 

This shift from $u$ to $X_u$ allows us to pass from the set $\Pb(I)$ equipped with the optimal transport metric $\Wuc(u,w)$ to the space $L^p([0,1])$ equipped with $\WXc(X_u,X_w)$ (defined in the Lemma below) which is equivalent to the $L^p$-distance of $X_u$ to $X_w$ as we can see by \aspb\ which tells us that $\frac{\alpha}{p} \abs{s}^p \leq \c(s)\leq \frac{\beta}{p} \abs{s}^p$. 

This shift is rigorously justified by the following Lemma that is an adaption of \cite[Thrm. 2.17.]{Santambrogio}:

\begin{lemma}
Let $\rho,\mu\in \Pb(\I)$ with densities $u,w$ such that $u+\frac{1}{u},w+\frac{1}{w}\in L^\infty(\I)$. Then the optimal transport distance $\Wuct(u,w)$ can be expressed in terms of IDFs $X_u,X_w$ as follows
\begin{align}
\Wuct(u,w)=\int_0^1 \c\left(\frac{X_u(\xi)-X_w(\xi)}{\tau}\right)\de \xi=:\WXct(X_u,X_w) \;.
\end{align}
\end{lemma}

So instead of gradient flows in the energy landscape of \eqref{def:u_energy_landscape} w.r.t. the optimal transport distance, we will consider the corresponding gradient flows of the inverse distribution functions in the energy landscape of 
\begin{align}
\HX(X):=\int_0^1 \hX(\partial_\xi X(\xi))\de \xi + \int_0^1 \vX(X(\xi))\de \xi
\end{align}
where $\hX(s)=s\hu(1/s)$. This corresponds to \eqref{def:u_energy_landscape} in the sense that $\Hu(u)=\HX(X_u)$ which can be checked easily when using that $X_u$ pushes the measure $\rho$ forward to the Lebesgue measure on $\J$. 

Instead of equation \eqref{PDE_u}, we will consider 
\begin{align}\label{eq:PDE_in_X}
\partial_t X= (c^*)^\prime(\partial_\xi  \dhX(\partial_\xi X) + \dvX(X))
\end{align}
which is the corresponding IDF-version of the \eqref{PDE_u}. The formal correspondence between \eqref{eq:PDE_in_X} and \eqref{PDE_u} can be seen easily, when one uses the push-forward connection between $u$ and $X_u$ which translates, since $X_u$ is differentiable a.e., to $u(x)=(\partial_\xi X_u(U_u(x)))^{-1}$.

\subsection{Discretization}
We will take care of the in-time discretization first. To that end, we will make use of the celebrated JKO-scheme, also known as the minimizing movement scheme (c.f. \cite{JKO1998}). Expressed in our case, we will approximate a solution of the equation with $X:\TJ\rightarrow \I$ by the piecewise constant in time interpolation 
\begin{align*}
X(t,\xi)=\Xn{n}(\xi) \quad \text{ when } t\in ((n-1)\tau,n\tau]\;.
\end{align*}
The sequence $\Xn{n}$ itself is given by the recursion 
\begin{align*}
\Xn{n}\in \argmin \WXct(\cdot, \Xn{n-1})+\HX(\cdot) 
\end{align*}
where the minimization takes place on the set of inverse distribution functions and $\Xn{0}$ is an approximation to the IDF of the initial data. 

The discretization in space will be a simple restriction to piecewise constant IDFs on a uniform grid $(0,\frac{1}{k},\frac{2}{k},\dots,1)$ on $\J$ and the values $x_i$ of such an $X$ on the grid will be encoded in a vector $\bx=(a=x_0,x_1,\dots, x_k=b)$. This way we receive by 
\begin{equation}\label{def:x_to_IDF}
\bx\mapsto X(\xi)= \sum_{i=0}^k x_i \ind{(\frac{i}{k},\frac{i+1}{k})}(\xi)
\end{equation}
a map that recovers $X$ from $\bx$. 

Since an interpretation of the IDF is, that for $\xi_1<\xi_2$ the value $\xi_2-\xi_1$ prescribes how much mass $u$ has on the interval $[X(\xi_1),X(\xi_2)]$, we recover our piecewise constant $u$ from the vector $\bx$ as 
\begin{align}\label{def:x_to_u}
\bx\mapsto u(x)=\sum_{i=0}^{k-1} \frac{1}{k(x_{i+1}-x_i)} \ind{(x_i,x_{i+1})}(x) \;.
\end{align}

Finally we define a replacement for the spatial derivative of $X$, which will be the piecewise constant function of finite forward differences given by the abbreviation $\delta x_i:=\frac{x_{i+1}-x_i}{1/k}$ 
\begin{align*}
\delta \bx = (\delta x_0,\delta x_1,\dots, \delta x_{k-1}) \mapsto \dexi X(\xi) =  \sum_{i=0}^{k-1} x_i \ind{(\frac{i}{k},\frac{i+1}{k})}(\xi)\;.
\end{align*}

A short explanation why we are dealing with piecewise constant IDF instead of the picewise affine linear ones, which would allow for a easier way to deal with the \glqq$\partial_\xi X$\grqq\ appearing in $\HX$, is in order. The main reason lies in the technical difficulties arising in the calculations to follow. The Euler-Lagrange equation will relate the approximate IDF $X$, which would be piecewise affine linear, and its derivatives, possibly of higher order, which would be piecewise constant at best. To bring these different forms together would require an unnecessary amount of technicalities that would all but divert from the results we want to discuss in this paper.

With this spatial discretization, we arrive at our fully discrete functionals that will be subject to minimization in each time-step.

The distance, the integral incorporating the external potential and the internal energy functional will take on the discrete form 
\begin{align*}
\tau \WXdctk(\bx,\bxn{n-1}) + \UXd(\bx) + \VXd(\bx) := \tau \dk \sum_{i=0}^k \c\left(\frac{1}{\tau}(x_i-x_i^{(n-1)})\right) + \dk \sum_{i=0}^{k-1} \hX\left(\delta x_i\right) + \dk\sum_{i=0}^k \vX(x_i)
\end{align*}
where $\UXd(\bx)$ will be $+\infty$ if $x_{i+1}=x_i$ for some $i$. Of course $\HXd(\bx)=\UXd(\bx)+\VXd(\bx)$ is the approximate energy functional applied to our vector $\bx$.


\subsection{Approximate solution and main theorem}
This leaves us with a recursive algorithm to receive a sequence $\bxn{n}$ from which we will recover an approximate solution to our problem \eqref{PDE_u}-\eqref{IniCond_u}. 

\begin{algorithm}
Let $k\in \N$ and $T,\tau_k>0$ such that $T/\tau_k\in \N$. Let $\bxn{0}_k \in \Xc_k(I)$ be some initial data. Then solve for $n=1,2,\dots,T/\tau_k$ the minimization problem 
\begin{align} \label{Minim_X} \begin{split}
\textit{minimize}\quad  & \bx\mapsto \tau_k \WXdctk(\bx,\bxn{n-1})+ \HXd(\bx) \\
\text{ on }\quad & \overline{\Xc_k(I)}=\{\bx\in \R^{k+1}\mid a=x_0\leq x_1\leq x_2\leq\dots\leq x_k=b\} \\
\textit{and set}\quad &\bx_k^{(n)}=\bx_{min} \;.
\end{split}
\end{align}
\end{algorithm}

We will abbreviate the functional that is minimized in each step as $\Phi(\tau_k;\bx^{(n-1)},\bx)$. The algorithm is well defined, as will be shown in the succeeding section. 

This way we receive a sequence $\bxn{n}_k$ of vectors inducing IDFs by \eqref{def:x_to_IDF} and densities by \eqref{def:x_to_u}. Indeed we will not show, that these  approximate densities converge to a solution of our PDE \eqref{PDE_u}-\eqref{IniCond_u} directly, but show that the approximate IDFs define a sequence of functions $X_k:\TJ\rightarrow \I$, converging to a solution of the PDE in terms of IDFs. 

Let us introduce in short the approximate IDFs first. Let $T=N_k\tau_k$. Then the piecewise constant in time and space functions $X_k$ are defined as 
\begin{equation}\label{def:piecewise_const_def}
X_k(t,\xi)=\sum_{n=0}^{N_k}\sum_{i=0}^{k-1} \xn{n}_i \ind{((n-1)\tau_k,n\tau_k]\times(\frac{i}{k},\frac{i+1}{k})}(\xi) \;.
\end{equation}
A more detailed introduction of the IDFs and other auxiliary functions is given in \textit{Definition \ref{def:Piecw_const_interp}}. 

To simplify the expressions appearing in the following claims and proofs, we have already abbreviate the finite forward difference quotient of our vectors $\bxn{n}$ as $\delta \xn{n}_i=\frac{\xn{n}_{i+1}-\xn{n}_{i}}{\dk}$ for $i=0,\dots,k-1$. 

Furthermore we now define the second symmetric difference quotient for a vector $\bx\in \Xc_k(\I)$ as 
\begin{align}\label{def:delta_squared}
\delta^2 \bx=(\delta^2 x_1,\dots,\delta^2 x_{k-1})\quad \text{ where } \quad \delta^2 x_i:=\frac{x_{i+1}-2x_i+x_{i-1}}{(\dk)^2} \quad \text{ for }i=1,\dots,k-1\;.
\end{align}
This way $\delta[\delta x_{i-1}]=\delta^2 x_i$ holds.

Note that we will denote with $p^\prime$ the H\"older conjugate $p^\prime=\frac{1}{1-1/p}$.

\begin{theorem}\label{Thm:Main_Thm}
Let $k\in \N$, $\tau_k>0$ be a sequence monotonically converging to zero and $\c$ a cost function satisfying \as --\aspb. 
Let $\bxn{0}_k$ be a sequence in $\Xc_k(I)$ with 
\begin{enumerate}
\item $X_k^{(0)}$, the IDFs corresponding to $\bxn{0}_k$, converge pointwise to some $X^{(0)}$ 
\item The energy $\HXd(\bxn{0}_k)$ is bounded
\item There are upper and lower bounds $\dxub>\dxlb>0$ such that the forward difference quotient $\delta \bxn{0}_k$ are bound from above and below by $\dxub$ and $\dxlb$ respectively.  
\end{enumerate}
Then the sequence $X_k$, defined by \eqref{def:piecewise_const_def}, has the following properties.
There is an unrelabled subsequence such that 
\begin{enumerate}
\item $X_k$ converges in $L^{p^\prime}(\TJ)$ to $X_\ast\in W^{1,\infty}(\TJoo)\cap C^{0,1/p}(0,T;L^1(\J))$ and additionally $\partial_\xi X_\ast\in L^1(0,T;W^{1,p^\prime}(\Jo))$.
\item The limit $X_\ast$ solves the following weak formulation of  \eqref{eq:PDE_in_X}: 
\begin{align*}
-\int_0^T \int_0^1 X_\ast(t,\xi)\partial_t \varphi(t,\xi)\de \xi \de t = \int_0^T \int_0^1 \dcs\left(\partial_\xi [ \dhX(\partial_\xi X_\ast(t,\xi) )]+\dvX(X_*(t,\xi) \right) \varphi(t,\xi) \de \xi \de t
\end{align*}
holds for all $\varphi\in C_c^\infty(\TJoo)$. 
\item The initial data are assumed continuously $\lim_{t\searrow 0} X_*(t)=X^{(0)}$ in $L^1(\J)$. 
\end{enumerate}
\end{theorem}

As already announced, if we assume more regularity of the initial data, the above claim also holds for the flux-limiting cost. 

\begin{theorem}\label{Thm:Scnd_Thm}
Let the prerequisites of \textit{Theorem \ref{Thm:Main_Thm}} hold, $\vu=\mathrm{const.}$ and $\c$ is not $p$-Wasserstein-like but flux-limiting. Additionally we assume finite bounds $\ddxub,\ddxlb$, such that the symmetric second difference quotient $\delta^2 \bxn{0}_k$ are bound from above and below by $\ddxub$ and $\ddxlb$ respectively. Then convergence to a weak solution of \eqref{eq:PDE_in_X} holds in the sense of \textit{Theorem \ref{Thm:Main_Thm}} with $p=2$ and the initial data are again assumed continuously. 
\end{theorem}


\subsection{Related results from the literature} The idea to use a Lagrangian scheme for a problem of the form \eqref{PDE_u} can be traced back to papers by MacCamy and Socolovski \cite{MacCamy1985Anumerical}, where $m=2$ and $\vu=0$ and Russo \cite{russo1990deterministic}, where $\hu(s)=s$ and again $\vu=0$. The former paper presents a discretization for the densities $u$ that is very similar to ours, the latter comparing Lagrangian discretizations and discussing generalizations to the two-dimensional case. In \cite{budd1999self} a moving mesh is used to capture numerically self similar solutions of the porous medium equation. In \cite{carrillo2009numerical} aggregation equations are studied in terms of a Lagrangian scheme basing on evolving diffeomorphisms.

The gradient flow structure of our equation was investigated by Agueh for $p$-Wasserstein cost \cite{Agueh} and by McCann and Puel for flux-limiting cost \cite{McCannPuel} both showing convergence of the minimizing movement scheme to a solution of the equation \eqref{PDE_u}. 

The connection between Lagrangian schemes and the gradient ﬂow structure was investigated by Kinderlehrer and Walkington \cite{kinderlehrer1999approximation} and in a series of unpublished theses \cite{roessler2004discretizing}, \cite{leven2002gradientenfluss}. Westdickenberg and Wilkening obtain in \cite{westdickenberg2010variational} a similar scheme as a by-product in the process of designing a structure preserving discretization for the Euler equations. Burger et al \cite{burger2009mixed} devise a numerical scheme in dimension two on basis of the gradient ﬂow structure, using the dynamic formulation of the Wasserstein distance \cite{benamou2000computational} instead of the Lagrangian approach. The Lagrangian approach was adapted to fourth order equations, namely by Cavalli and Naldi \cite{cavalli2010wasserstein} for the Hele-Shaw ﬂow, and by Düring et al \cite{during2010gradient} for the DLSS equation.

The spatial discretization we will be analyzing was proposed by Laurent Gosse and Giuseppe Toscani and analysed in \cite{gosse2006identification}. It is worth noting that therein the approach taken in our paper, to solve the equation in terms of IDFs, was taken as well, as was in \cite{DanielBeni} to analyse a aggregation-advection-diffusion equation. In \cite{DanielHorst} convergence  for a family of energies including a potential energy and more general $\hu$ in the special case quadratic Wasserstein distance. 

We want to mention another numerical scheme close to the one examined in this paper, but with the possibility for generalization to higher dimensions. Instead of decomposing the transport map $G$, we could have discretized the diffeomorphism $G$ itself. This approach was recently thoroughly investigated by Carillo et al \cite{CarilloMatthes}. 

\subsection{Outline}
Proving \textit{Theorem \ref{Thm:Main_Thm}} will basically consist of three steps. 

First, we will show some properties about the sequence of minimizers $\bxn{n}_k$: it is well defined, it satisfies a maximum/minimum principle, we will give a discrete Euler-Lagrange equation corresponding to the minimization problem
and finally we will establish a priori estimates 
which gives us control over the right hand side of the Euler-Lagrange equation. 

Then we will prove the main convergence results, especially the strong convergence of $ \dexi X_k$, the weak compactness of the sequence of functions $\CdelH_k$ corresponding to the r.h.s. of the Euler-Lagrange equation and consequently of $\detau X_k$ in $L^p(\TJ)$. 

Another convergence result that we have to established is the limit of the Euler-Lagrange equation in terms of functions. However, since $\dcs$ is not linear, we cannot identify the limit of the r.h.s. $\CdelH_k$ right away to receive the equation that is satisfied by $X_*$. 

This final step of identifying the limit will be done by using a monotonicity argument, the so called Browder Minty trick. After the identification of the non-linear limit, we have shown the convergence of our sequence $X_k$ to a limit $X_\ast$ that satisfies the weak formulation of the PDE in terms of IDF. We will combine this with the in-time regularity of $X_*$, which is H\"older-continuous on $\T$ for $\alpha=1/p$, to receive continuity at $t=0$. 

Concerning \textit{Theorem \ref{Thm:Scnd_Thm}}, the main reason we can apply the arguments for the $p$-Wasserstein cost to the flux-limiting cost is a second maximum-/minimum principle, this time for $\delta^2 \bxn{n}$, which shows that if we start with initial data that are regular enough, then our minimization problem in some sense does not see the discontinuity of our cost and the cost can be treated as $2$-Wasserstein like. We will note the key parts of the proofs where a difference has to be made if one is working with flux-limiting cost. 

In the last part of the paper, we give some numerical examples to illustrate the dynamics of the $p$-Wasserstein diffusion and flux-limiting diffusion. Additionally there will be an approximate solution to a parabolic $q$-Laplace equation and some numerical convergence analysis.

\section{Properties of the minimization problem}
In this section we want to lay the foundation for the succeeding section, which will show the crucial convergences that are needed for the main theorem. First of all, we will show that the sequences $\bxn{n}$ are indeed well defined. This is followed by the Euler-Lagrange equation for the minimization. Additionally we will show two estimates.

\subsection{Existence and uniqueness}

\begin{lemma}
The minimization problem \eqref{Minim_X} has a unique minimizer $\bx=(x_0,x_1,\dots,x_k)\in \Xc_k(\I)$.
\end{lemma}

\begin{proof}
The set $\Xc_k(\I)$ is bounded in $\R^{k+1}$ and $\Phi(\tau;\bx^{(n-1)},\bx)$ is lower semi-continuous in $\bx$. $\bx=\bx^{(n-1)}$ generates a finite value in $\Phi$, so there exists a minimizer in $\overline{\Xc_k(\I)}$ with finite value. Additionally, the functional is strictly convex, since both summands are and $\tau>0$, so the minimizer is unique.
Note that $\hX$ is monotonously decreasing. This yields the following inequalities with $\inf_{\bx\in \Xc_k(\I)} \VXd(\bx)=:\VXdb>-\infty$ (which holds by properties of $\Xc_k(\I)$ and $\vX$) for $j=0,\dots,k-1$:
\begin{align}
\begin{split}
C\geq &\tau_k \WXdctk(\bx_{min},\bx_k^{(n-1)})+ \HXd(\bx_{min}) \\
&\geq \dk \sum_{i=0}^{k-1}\hX(\delta \underline{x}_i)+ \VXdb \\
&\geq \dk \hX(\delta \underline{x}_j) + \dk(k-1)\hX\left(\frac{b-a}{d_k}\right)+\VXdb\;.
\end{split}
\end{align}
Therefore $\hX\left(\frac{\underline{x}_{j+1}-\underline{x}_j}{\dk}\right)$ in $\HXd$ is bound from above for every $j=0,\dots,k$. Finally, by $\lim_{s\searrow 0} \hX(s)=\infty$ we conclude $\underline{x}_{j+1}-\underline{x}_j\geq\varepsilon$ for some $\varepsilon>0$ which shows $\bx_{min}\in \Xc_k(\I)$. 
\end{proof}

Now that we know that the minimization is well defined, we can establish the following result.
\begin{corollary}\label{cor:sequence_descends_in_H}
Let $\bxn{0}\in \Xc_k(\I)$ and $\bxn{n}$ with $n=1,2,\dots$ the sequence of minimizer recursively defined by \eqref{Minim_X}. Then 
\begin{align}
\WXdctk(\bxn{n},\bxn{n-1})\leq \HXd(\bxn{n-1})-\HXd(\bxn{n})
\end{align}
holds for every $n=1,2,\dots$.
\end{corollary}

\begin{proof}
Plugging in the feasible $\bx=\bxn{n-1}$ in the minimization problem combining it with $\bxn{n}$ being a minimizer and with $\WXdctk(\bxn{n-1},\bxn{n-1})=0$ yields the result after rearranging.
\end{proof}

\subsection{The Euler-Lagrange equation}
Formulating the Euler-Lagrange equation n the case of $p$-Wasserstein cost is straight forward. In the case of flux-limiting cost, however, we have to make sure the minimization problem in some sense does not see the discontinuity.

\begin{lemma}
In the case of flux-limiting cost, the minimizer $\bxn{n}$ lies in 
\[\left\{\bx \middle| \bx\in \Xc_k(\I), \abs{x_j-\xn{n-1}_j}<\ls\tau\text{ for }j=0,\dots,k\right\}\;.\] 
\end{lemma}

\begin{proof}
Let $P:=\left\{i\in \{0,\dots,k\}\mid \abs{\xn{n}_j-\xn{n-1}_j}=\ls\tau\right\}$ and w.l.o.g. $\abs{\xn{n}_j-\xn{n-1}_j}\leq\ls\tau$. Define a partial convex combination of $\bxn{n}$ and $\bxn{n-1}$ w.r.t. $P$ as 
\[\bx^\varepsilon:=\begin{cases}
(1-\varepsilon) \xn{n}_i+\varepsilon\xn{n-1}_i&\text{ if }i\in P \\
\xn{n}_i &\text{ if }i\notin P\;.
\end{cases} \]
We will now show that if $P\neq \emptyset$, then for a suitable $\varepsilon$ the partial convex combination $\bx^\varepsilon$ is a feasible candidate with $\Phi(\tau,\bxn{n-1},\bx^\varepsilon)<\Phi(\tau,\bxn{n-1},\bxn{n})$ contradicting $\bxn{n}$ being a minimizer in the first place.

We note that $\bx^\varepsilon \in \Xc_k(\I)$ for $\varepsilon$ small enough, since $\bxn{n}\in \Xc_k(\I)$ and $\Xc_k(\I)$ is open w.r.t. $\{a\}\times \R^{k-1}\times \{b\}$. 

Furthermore recall $\dc(s)\rightarrow \pm \infty$ for $s\xrightarrow{s\in (-1,1)} \pm 1$. Let  
\[ P^+:=\{i\in \{0,\dots,k\}\mid i\notin P, i+1\in P\} \quad \text{ and }\quad P^-:=\{i\in \{0,\dots,k\}\mid i\in P, i+1\notin P\}\;. \]
Recall $c,\hX$ and $\vX$ are strictly convex $C^1$-functions, which allows us to calculate
\begin{align*}
\Phi(\tau;&\bxn{n-1},\bxn{n})-\Phi(\tau;\bxn{n-1},\bx^\varepsilon) \\
 &=\dk \sum_{i=0}^k \c\left(\frac{1}{\tau}(x_i^{(n)}-x_i^{(n-1)})\right)-\c\left(\frac{1}{\tau}(x_i^\varepsilon-x_i^{(n-1)})\right)\\ 
 &\quad + \dk \sum_{i=0}^{k-1} \hX\left(\frac{\xn{n}_{i+1}-\xn{n}_i}{\dk}\right)-\hX\left(\frac{x^\varepsilon_{i+1}-x^\varepsilon_i}{\dk}\right)  \\
 &>  \varepsilon\dk \Bigg(\sum_{i\in P} \c^\prime\left(\frac{1}{\tau}(x_i^\varepsilon-x_i^{(n-1)})\right) \frac{1}{\tau}\left(\xn{n}_i-\xn{n-1}_i\right) + \dk \sum_{i\in P\setminus P^-} \dhX\left(\frac{x^\varepsilon_{i+1}-x^\varepsilon_i}{\dk}\right) \left(\delta \xn{n}_i-\delta \xn{n-1}_i\right) \\
  &\quad  + \dk \sum_{i\in P^-} \dhX\left(\frac{\xn{n}_{i+1}-x^\varepsilon_i}{\dk}\right) k \left( \xn{n-1}_i- \xn{n}_i\right) + \dk \sum_{i\in P^+} \dhX\left(\frac{x_{i+1}^\varepsilon - \xn{n}_i}{\dk}\right) k \left( \xn{n}_i- \xn{n-1}_i\right)  \Bigg)\;.
\end{align*}
By continuity, the expressions $\dhX\left(\frac{x^\varepsilon_{i+1}-x^\varepsilon_i}{\dk}\right)$, $\dhX\left(\frac{\xn{n}_{i+1}-x^\varepsilon_i}{\dk}\right)$ $\dhX\left(\frac{x_{i+1}^\varepsilon - \xn{n}_i}{\dk}\right)$ and $\dhX\left(\frac{x_{i+1}^\varepsilon - \xn{n}_i}{\dk}\right)$ are bounded in $\varepsilon$ for $\varepsilon\searrow 0$. On the other hand, $\dc\left(\frac{1}{\tau}(x_i^\varepsilon-x_i^{(n-1)})\right) \left(\xn{n}_i-\xn{n-1}_i\right)\rightarrow \infty$ for $\varepsilon\searrow 0$. So we know that, for $\varepsilon$ small enough, $\Phi(\tau;\bxn{n-1},\bxn{n})-\Phi(\tau;\bxn{n-1},\bx^\varepsilon)$ will be positive, which is our sought for contradiction.
\end{proof}

\begin{lemma}[Euler-Lagrange equation]\label{lem:Euler-Lag_disc}
Let $\bxn{n-1} \in \Xc_k(\I)$. Let $\bxn{n}$ be the corresponding minimizer in \eqref{Minim_X}. Then it satisfies the system of equations  
\begin{align}\label{Euler-Lagrange}
\frac{\xn{n}_i-\xn{n-1}_i}{\tau_k}=\dcs\left(\frac{\dhX\left( \frac{\xn{n}_{i+1}-\xn{n}_{i}}{\dk} \right)-\dhX\left( \frac{\xn{n}_i-\xn{n}_{i-1}}{\dk} \right)}{\dk}-\dvX(\xn{n}_i)\right) 
\end{align}
for each $i=1,\dots,k-1$. Note that the above equation reduces to 
$ \frac{\xn{n}_i-\xn{n-1}_i}{\tau_k}=0 $
for $i\in\{0,k\}$.
\end{lemma}

\begin{remark}\label{rem:Aink}
We will abbreviate the argument of $\dcs$ above:
\begin{align}\label{eq:Aink_abbreviation}
\Aink:=\frac{\dhX\left( \delta \xn{n}_{i} \right)-\dhX\left( \delta \xn{n}_{i-1} \right)}{\dk}-\dvX(\xn{n}_i)\;.
\end{align} 
\end{remark}

\begin{proof}
The functional $\Phi(\tau_k;\bxn{n-1},\cdot)$ is continuously differentiable on $\Xc_k(\I)=\{\bx\mid a=x_0<x_1<\dots<x_k=b\}$ 
with gradient 
\begin{align}\label{eq:Eul-Lag-Intermediate}
\partial_{x_j} \Phi(\tau;\bxn{n-1},\bx)= \dk \dc\left(\frac{1}{\tau}(x_j-\xn{n-1}_j)\right)+\dhX(\delta x_{j-1})-\dhX(\delta x_{j})+ \dk \dvX(x_j)\;.
\end{align}
We receive \eqref{Euler-Lagrange} as the first order optimality conditions, using $\dcs=(\dc)^{-1}$.
\end{proof}

\subsection{The discrete minimum/maximum principle}
We will prove the first discrete minimum/maximum principle next. It bounds the forward difference quotient of $\bxn{n}_k$ uniformly from above and away from zero, if we have initial data as described in \textit{Theorem \ref{Thm:Main_Thm} (4)}. This initial condition corresponds to $u_0$ being bound from above and away from zero.  

\begin{lemma}\label{lem:min_max_princ}
Let $\bxn{n-1}\in \Xc_k(\I)$ and $\overline{M}>\underline{M}>0$ such that $\underline{M}\leq \delta \xn{n-1}_i\leq \overline{M}$ holds for $i=0,\dots,k-1$. Let $\bxn{n}$ be the minimizer of $\Phi(\tau;\bxn{n-1},\bx)$. 
\begin{enumerate}
\item If $\vu=\mathrm{const.}$, then $\underline{M}\leq \delta \xn{n}_i\leq \overline{M}$ holds for $i=0,\dots,k-1$. 
\item If \asv\ holds, then $\eexp{-\tau_k\kappa} \underline{M}\leq \delta \xn{n}_i\leq \overline{M}$ holds for $i=0,\dots,k-1$, where $\kappa:=\frac{\ddvXa}{\ddcb}$ does only depend on $\I,\vX$ and $\c$ and is given by $\ddvXa:=\max_{x\in[a,b]} \ddvX(x)<\infty$ and $\ddcb$ which is defined in \asv . 
\end{enumerate}
\end{lemma}

\begin{remark}
In the second case, there is still a uniform lower bound on $\delta \xn{n}_i$ for finite time horizons $T$. Indeed let $N_k\tau_k=T$. Then for all $i,n$ with $n\leq N_k$ we have $\eexp{-T\kappa} \underline{M}\leq \delta \xn{n}_i$. Note that this lower bound $\eexp{-T\kappa} \underline{M}$ does only depend on $\dxlb$, $\c$, $\vX$, $T$.
\end{remark}

\begin{proof} 
The proof will consist in both cases of considering an index $i$ such that $\delta \xn{n}_i\geq \delta \xn{n}_j$ for all $j\in \{0,\dots,k-1\}$ (or $\leq$ for the bound from below) and then showing that the r.h.s. of the Euler-Lagrange equation is non-positive for that $i$, implying that $\delta \xn{n}_i\leq \delta \xn{n-1}_i$ which entails the claim. 

For the bound from below in the second claim the calculations will be a little more involved and we can only show $(1+C)\delta \xn{n}_i\leq \delta \xn{n-1}_i$ which then results in the special form of the lower bound in that case. 

In any of the cases, the first step will consist of the use of Taylor's theorem on $\dcs$ in the r.h.s. of the Euler-Lagrange equation, after applying the forward difference $\delta$ to the equation, to receive 
\begin{align*}
\frac{\delta\xn{n}_i-\delta\xn{n-1}_i}{\tau_k}=\delta\left[\dcs\left( \Aink \right)\right] = \ddcs(\zeta_i)\cdot \delta \Aink\;.
\end{align*} 

\begin{enumerate}
\item We will show representatively the bound from above. The bound from below can be obtained by nearly the same calculations.

So let $i\in \{0,\dots,k-1\}$ be an index such that $\delta \xn{n}_i\geq \delta \xn{n}_j$ for all $j\in \{0,\dots,k-1\}$. Since $\dhX$ is monotonously increasing, $\dhX(\delta \xn{n}_i)\geq \dhX(\delta \xn{n}_{j})$ holds for all $j\in \{0,\dots,k-1\}$, too. Consequently, the the second symmetric difference quotient $\delta^2\dhX(\delta \xn{n}_i)$ of $\dhX(\delta \xn{n}_i)$ is non-positive. 
By $\dvX=0$ we can rewrite $\delta \Aink= \delta^2 \dhX(\delta \xn{n}_i)$ and receive
\begin{align*}
\ddcs(\zeta_i)\cdot \delta \Aink &= -\ddcs(\zeta) \delta^2\dhX(\delta \xn{n}_i) \leq 0
\end{align*}
With this inequality at hand we can conclude
\[ \delta \xn{n}_j \leq \delta \xn{n}_i\leq \delta \xn{n-1}_i \leq \overline{M}  \]
for all $j\in \{0,\dots,k-1\}$.

\item Let $\dvX\neq 0$, but \asv\ holds. Then 
\begin{align*}
\delta \Aink &= \delta^2 \dhX(\delta \xn{n}_i) - \delta [\dvX(\xn{n}_i)] \\
&= \delta^2 \dhX(\delta \xn{n}_i) -\ddvX(\hat \zeta_i) \delta \xn{n}_i\;.
\end{align*}
Since $\ddvX$ is continuous on $\I$ and $\vX$ is convex, we have can bound $0\leq \ddvX(\hat \zeta_i)\leq \ddvXa$. Additionally, since $\dcs$ is monotone and since \asv\ holds we can furthermore bound $0\leq \ddcs(\zeta_i)\leq \frac{1}{\ddcb}$. 

Indeed by  $\ddc \geq \ddcb >0$ the factor $\ddcs(\zeta)$ can be bound from above, independently of $\zeta$ as follows. 
\begin{align*}
\ddcs(\zeta)&= \left((\dc)^{-1}\right)^\prime(\zeta) =\frac{1}{\ddc\big( \dcs(\zeta) \big) } \leq \frac{1}{ \ddcb}\;.
\end{align*}
Now assume $\delta \xn{n}_i\geq \delta \xn{n}_{j}$ for all $j\in \{0,\dots, k-1\}$. Then, again, $\ddcs(\zeta_i)\cdot \delta^2 \dhX(\delta \xn{n}_i)\leq 0$. Furthermore $\ddcs(\zeta_i)\cdot \ddvX(\hat \zeta_i) \delta \xn{n}_i \geq 0\cdot \ddvX(\hat \zeta_i) \delta \xn{n}_i= 0$ since $\ddvX(\hat \zeta_i) \delta \xn{n}_i \geq 0$. Then again, this implies 
\[ \delta \xn{n}_i \leq \delta \xn{n-1}_i \]
and we have shown the claim. 

Let now, on the other hand, $\delta \xn{n}_i\leq \delta \xn{n}_{j}$. Then, by an analogous argument as above, $\ddcs(\zeta_i)\cdot \delta^2 \dhX(\delta \xn{n}_i)\geq 0$. But this time, 
\[ \ddcs(\zeta_i)\cdot \ddvX(\hat \zeta_i) \delta \xn{n}_i \leq \frac{1}{\ddcb} \ddvXa \delta \xn{n}_i \]
giving us only 
\begin{align*}
\delta \xn{n-1}_j-\delta \xn{n}_j \leq  \tau   \frac{\ddvXa}{\ddcb}  \delta \xn{n}_j   \;.
\end{align*}  

This inequality implies by $(1+x)\leq \eexp{x}$
\[ \min_{i} \delta \xn{n-1}_i \leq \delta \xn{n-1}_j \leq \left(1+\tau   \frac{\ddvXa}{\ddcb} \right)  \delta \xn{n}_j \leq \eexp{\tau \kappa} \delta \xn{n}_j =\eexp{\tau \kappa} \min_{i} \delta \xn{n}_i  \]
where we abbreviated $\kappa:=\frac{\ddvXa}{\ddcb}$. This proves the lower bound. \qedhere
\end{enumerate}
\end{proof}

We will show the second minimum/maximum principle next. Its claim, expressed in terms of the probability density $u$, is as follows. Assume we start with initial data $u_0$ that are bounded from above and away from zero and the weak derivative of $u_0$ is bounded. Then the weak spatial derivative of the solution $u$ of the PDE \eqref{PDE_u} will be bounded as well. 

\begin{lemma}\label{lem:min_max_2}
Let the prerequisites of \textit{Theorem \ref{Thm:Scnd_Thm}} hold and $\bxn{n-1}\in \Xc_k(\I)$ and let $\bxn{n}$ be the minimizer of $\Phi(\tau;\bxn{n-1},\bx)$. Then 
\[\min \{\min_{j} \delta^2 \xn{n-1}_j,0\}\leq \delta^2 \xn{n}_i\leq \max \{\max_{j} \delta^2 \xn{n-1}_j,0\}\] holds for $i=1,\dots,k-1$. 
\end{lemma}

Note that with initial data $\bxn{0}_k$ bounded as in \textit{Theorem \ref{Thm:Scnd_Thm}}, this implies the bounds $\min\{\ddxlb,0\}\leq \delta^2 \xn{n}_i \leq \max \{\ddxub,0\}$ for every element $\delta^2 \xn{n}_i$ of every $\delta^2 \bxn{n}_k$ i.e. for all $i,n,k$.

\begin{proof}
We will, again, exemplarily derive the upper bound and the lower one can be found by very similar calculations. 

This proof is a little bit more elaborate than the proof of \textit{Lemma \ref{lem:min_max_princ}}. We will again start by assuming $\delta^2 \xn{n}_i\geq \delta^2 \xn{n}_j$ for all $j\in\{1,\dots,k-1\}$ but this time we do not a priori know that $\delta^2\xn{n}_i$ is positive, which is the reason for the structure of the bounds. Additionally it does not appear explicitly on the right hand side of the Euler-Lagrange equation but is hidden in
\[ \delta[\dhX(\delta\xn{n}_{i-1})]=\frac{\dhX\left(\delta \xn{n}_i\right)-\dhX\left( \delta \xn{n}_{i-1} \right)}{\dk} \;. \] 

So let us consider $\delta^2 \xn{n}_i\geq \delta^2 \xn{n}_j$ for all $\{1,\dots,k-1\}$ and lets furthermore assume $\delta^2 \xn{n}_i>0$ since otherwise there is nothing left to show. Then we receive as the second symmetric difference quotient of the Euler-Lagrange equation
\begin{align*} 
\frac{\delta^2 \xn{n}_i-\delta^2 \xn{n-1}_i}{\tau_k}&=\delta^2 \left[\dcs\left(\delta [\dhX(\delta \xn{n}_{i-1})]\right) \right] \\
&= \frac{ \dcs\left(\delta [\dhX(\delta \xn{n}_{i})]\right)
-2\dcs\left(\delta [\dhX(\delta \xn{n}_{i-1})]\right)
+\dcs\left(\delta [\dhX(\delta \xn{n}_{i-2})]\right)  }{(\dk)^2}
\end{align*}
and by applying Taylor's theorem at the point $\delta [\dhX(\delta \xn{n}_{i-1})]$ to the first and last summand the numerator equals 
\begin{align*}
\ddcs(\zeta_i) \left( \delta [\dhX(\delta \xn{n}_{i})]-\delta [\dhX(\delta \xn{n}_{i-1})]\right) + \ddcs(\zeta_{i-1})\left( \delta [\dhX(\delta \xn{n}_{i-2})]-\delta [\dhX(\delta \xn{n}_{i-1})]\right)\;.
\end{align*}
We will show 
\[ \left( \delta [\dhX(\delta \xn{n}_{i})]-\delta [\dhX(\delta \xn{n}_{i-1})]\right) <0 \]
next. To that end, note that our assumptions imply $\delta \xn{n}_i >\delta \xn{n}_{i-1}$ and consider the following two cases. 
\begin{itemize}
\item[Case 1:] $\delta \xn{n}_{i+1}>\delta \xn{n}_i$. This implies $\delta^2 \xn{n}_{i+1}>0$. We apply Taylor's theorem again, to receive 
\begin{align*}
\delta [\dhX(\delta \xn{n}_{i})]-\delta [\dhX(\delta \xn{n}_{i-1})]&= \ddhX(\hat \zeta_i)\delta^2 \xn{n}_{i+1} - \ddhX(\hat \zeta_{i-1}) \delta^2 \xn{n}_{i}
\end{align*}
where $\hat \zeta_i>\delta \xn{n}_i>\hat \zeta_{i-1}$ and therefore $\ddhX(\hat \zeta_i)<\ddhX(\hat \zeta_{i-1})$ since $\ddhX$ is strictly dicreasing. Consequently 
\[\ddhX(\hat \zeta_i)\delta^2 \xn{n}_{i+1} - \ddhX(\hat \zeta_{i-1}) \delta^2 \xn{n}_{i} < \ddhX(\hat \zeta_{i-1})\left( \delta^2 \xn{n}_{i+1} -  \delta^2 \xn{n}_{i} \right) \leq 0 \]
by our assumption on $\delta^2 \xn{n}_i$ and $\ddhX>0$.
\item[Case 2:] $\delta \xn{n}_{i+1}<\delta \xn{n}_i$. Together with $\delta \xn{n}_{i-1}<\delta \xn{n}_i$ this implies, by monotonicity of $\dhX$ that 
\[ \delta [\dhX(\delta \xn{n}_{i})]-\delta [\dhX(\delta \xn{n}_{i-1})]=\frac{\dhX(\delta \xn{n}_{i+1})-2\dhX(\delta \xn{n}_{i})+\dhX(\delta \xn{n}_{i-1})}{\dk} \leq 0 \]
holds, right away.
\end{itemize}
With a similar argument one can prove that 
\[ \delta [\dhX(\delta \xn{n}_{i-2})]-\delta [\dhX(\delta \xn{n}_{i-1})]< 0 \]
holds, too. 

Now $\ddcs$ being positive implies 
\[ \ddcs(\zeta_i) \left( \delta [\dhX(\delta \xn{n}_{i})]-\delta [\dhX(\delta \xn{n}_{i-1})]\right) + \ddcs(\zeta_{i-1})\left( \delta [\dhX(\delta \xn{n}_{i-2})]-\delta [\dhX(\delta \xn{n}_{i-1})]\right) < 0 \]
and consequently 
\[ \frac{\delta^2 \xn{n}_i-\delta^2 \xn{n-1}_i}{\tau_k} < 0 \]
which then implies the claim. 
\end{proof}

\subsection{A priori estimates}

\begin{proposition}\label{prop:energy_inequality_1}
Let $\bxn{0}\in \Xc_k(\I)$ and $\bxn{n}$ $n=1,2,\dots$ the sequence of minimizer recursively defined by \eqref{Minim_X} and let us abbreviate  $\ct(s):=\dc(s)s$ (c.f. \aspb ). Then, for every $n$ and every pair of indices $m_1<m_2$  
\begin{align}\label{ineq:Energy_Sum}
\begin{split}
\tau_k \dk \sum_{i=0}^k \tilde \c \left( \frac{\xn{n}_i-\xn{n-1}_i}{\tau_k}\right) &\leq \HXd(\bxn{n-1})-\HXd(\bxn{n})\\
\tau_k \dk \sum_{n=m_1+1}^{m_2} \sum_{i=0}^k \ct\left( \frac{\xn{n}_i-\xn{n-1}_i}{\tau_k}\right) &\leq \HXd(\bxn{m_1})-\HXd(\bxn{m_2})
\end{split}
\end{align}
holds.
\end{proposition}

\begin{proof}
Recall $\delta \xn{n}_i=\frac{\xn{n}_{i+1}-\xn{n}_i}{\dk}$. The equation \eqref{eq:Eul-Lag-Intermediate} has to be equal to zero at a minimizer implying for $i=1,\dots,k-1$
\begin{align*}
\tau_k \dk \c\left( \frac{\xn{n}_i-\xn{n-1}_i}{\tau_k}\right)=-\left[\dhX\left( \delta \xn{n}_{i-1} \right)-\dhX\left( \delta \xn{n}_{i} \right)+ \dk \dvX(\xn{n}_i)\right]\;.
\end{align*}
Multiplying these equations by $(\xn{n}_i-\xn{n-1}_i)$ and summing them up gives us, after rearranging and an index-shift as well as using convexity of $\hX$
\begin{align*}
\tau_k \dk \sum_{i=1}^{k-1}& \tilde \c\left( \frac{\xn{n}_i-\xn{n-1}_i}{\tau_k}\right) \\ 
&\leq \dk \sum_{i=1}^{k-1} \hX\left( \delta \xn{n-1}_i \right)-\hX\left( \delta \xn{n}_i \right)  +  \vX(\xn{n-1}_i)-\vX(\xn{n}_i) \\
&= \HXd(\bxn{n-1})-\HXd(\bxn{n})\;.
\end{align*}
This shows the first claim. The second one follows directly when summing up the first one from $n=m_1+1$ to $m_2$ and minding the telescopic sum on the r.h.s. .
\end{proof}

\begin{corollary}\label{cor:Hoel_Est}
The estimates \eqref{ineq:Energy_Sum} imply H\"older-type estimates. Let $t_1=m_1\tau_k$ and $t_2=m_2\tau_k$. Then
\begin{align}
\dk \sum_{i=0}^{k} \abs{\xn{m_2}_i - \xn{m_1}_i}  \leq (t_1-t_2+\tau_k)^{1/p} \frac{1}{\sqrt[p]{\alpha}}\left( \HXd(\bxn{0})-\underline \HXd \right)^{1/p}
\end{align}
holds for some $C$ depending only on $m$, $\HXd(\xn{0})$ and $\c$. 
\end{corollary}

\begin{proof} Recall \aspb\ implies $\abs{\frac{\xn{n}_i - \xn{n-1}_i}{\tau_k}}^p\leq \frac{1}{\beta}\ct \left(\frac{\xn{n}_i - \xn{n-1}_i}{\tau_k}\right)$.
A standard H\"older type argument shows 
\begin{align*}
\dk \sum_{i=0}^{k} \abs{\xn{m_2}_i - \xn{m_1}_i}
&\leq (t_2-t_1+\tau_k)^{1/p^\prime} \left( \tau_k \frac{1}{\alpha} \dk \sum_{n=m_1+1}^{m_2}   \sum_{i=0}^{k} \ct \left(\frac{\xn{n}_i - \xn{n-1}_i}{\tau_k}\right)\right)^{1/p} \\
&\leq (t_2-t_1+\tau_k)^{1/p^\prime} \frac{1}{\sqrt[p]{\alpha}}\left( \HXd(\bxn{m_1})-\HXd(\bxn{m_2}) \right)^{1/p} \;.
\end{align*}
Now with \textit{Corollary \ref{cor:sequence_descends_in_H}} we can estimat $\HXd(\bxn{m_1})\leq \HXd(\bxn{0})$ for every $n=0,1,\dots$. And since $\HXd$ is bound from below by $\HXdb>-\infty$, where $\HXdb$ does not depend on $k$ (and therefore not on $\tau_k$ either) we arrive at our claim with $\HXd(\bxn{m_2})\geq \HXdb$.
\end{proof}

\begin{proposition}\label{prop:Entropy_ineq_conseq}
For every $k,N$, the sequence $(\bxn{n}_k)_{n=1,\dots,N}$ fulfills the following Entropy-inequality
\begin{align*}
\tau_k \dk  \sum_{n=1}^N \sum_{i=1}^{k-1} \abs{\Aink }^{p^\prime}&\leq \sqrt[p-1]{\beta} \left(\sup_{k}\HXd(\bxn{0}_k)- \HXdb\right) \\
\tau_k \dk  \sum_{n=1}^N \sum_{i=1}^{k-1} \abs{\dcs\left(\Aink\right) }^p&\leq \frac{1}{\alpha} \left(\sup_{k}\HXd(\bxn{0}_k)- \HXdb\right)\;.
\end{align*}
Note that these bounds do not depend on $k$.
\end{proposition}

\begin{proof}
The starting point of this proof will be the second inequality in \eqref{ineq:Energy_Sum}. We will show an equation for the r.h.s. and consequently the upper bound will apply to $\Aink$ as well. 

Let $T=N\tau_k$. Then we receive when applying the Euler-Lagrange equation in the last step
\begin{align*}
\tau_k \dk \sum_{n=1}^N \sum_{i=1}^{k-1} \ct \left( \frac{\xn{n}_i-\xn{n-1}_i}{\tau_k} \right)&=\tau_k \dk \sum_{n=1}^N \sum_{i=1}^{k-1}  \c\left( \frac{\xn{n}_i-\xn{n-1}_i}{\tau_k} \right)\left( \frac{\xn{n}_i-\xn{n-1}_i}{\tau_k} \right) \\
&= \tau_k \dk \sum_{n=1}^N \sum_{i=1}^{k-1} \Aink \dcs\left(\Aink\right)\;.
\end{align*}
Now \textit{Proposition \ref{prop:energy_inequality_1}} yields
\begin{align*}
\tau_k \dk \sum_{n=1}^N \sum_{i=1}^{k-1} \Aink \dcs\left(\Aink\right)=\tau_k \dk \sum_{n=1}^N \sum_{i=1}^{k-1}& \ct \left( \frac{\xn{n}_i-\xn{n-1}_i}{\tau_k} \right) \leq \sup_{k}\HXd(\bxn{0}_k)- \HXdb\;.
\end{align*}
The expression above on the l.h.s. invites us to use the estimate from \aspb\ $\alpha \abs{s}^p \leq s\dc(s)\leq \beta \abs{s}^p$. On the one hand $r\dcs(r)\leq \frac{1}{\sqrt[p-1]{\alpha}}\abs{r}^{p^\prime}$ and on the other hand $r\dcs(r) \geq \frac{1}{\sqrt[p-1]{\beta}}\abs{r}^{p^\prime}$ so we arrive at  
\begin{align*}
\tau_k \dk  \sum_{n=1}^N \sum_{i=1}^{k-1} \abs{\Aink }^{p^\prime}&\leq \sqrt[p-1]{\beta} \left(\sup_{k}\HXd(\bxn{0}_k)- \underline{\HXd}\right) \\
\tau_k \dk  \sum_{n=1}^N \sum_{i=1}^{k-1} \abs{\dcs \left(\Aink\right) }^p&\leq \frac{1}{\alpha} \left(\sup_{k}\HXd(\bxn{0}_k)- \HXdb\right)
\end{align*}
immediately.
\end{proof}

\section{Convergences of the approximate solution}
In this section we prove the convergence of the Euler-Lagrange equation in terms of functions and obtain a convergence result needed in the next section when it comes to identifying the limit of the Euler-Lagrange equation. 

Before we can talk about convergences in terms of functions, we have to define our approximate IDF and other approximate functions.

\subsection{The approximate solutions} 
The approximations to solutions of our PDE \eqref{PDE_u} in terms of inverse distribution functions will be defined as piecewise constant using the sequences of vectors $\bxn{n}_k$ from above. To that end we define a general rule to receive functions from such a sequence ${\bf b}:=(b_i^{(n)})_{i=0,\dots,k \atop n=0,\dots,N_k}$. 

\begin{definition}\label{def:Piecw_const_interp}
Let $\tau_k \searrow 0$ in $k$. Let furthermore $T=N_k\tau_k$ and $(b_i^{(n)})_{i=0,\dots,k \atop n=0,\dots,N_k}$ be a real valued sequence. We will abbreviate  $\IndTX(t,\xi):=\ind{((n-1)\tau_k,n\tau_k]\times (\tfrac{i-1}{k},\tfrac{i}{k}]}(t,\xi)$.

Then we define the $b_i^{(n)}$-induced, function $\mathfrak{P}[{\bf b}]:\TJ\rightarrow \R$ as the piecewise constant function on $((n-1)\tau_k,n\tau_k]\times (\tfrac{i-1}{k},\tfrac{i}{k}]$ as 
\[ \mathfrak{P}[{\bf b}] (t,\xi):=  \sum_{n=0}^{N_k} \sum_{i=1}^k b^{(n)}_i \IndTX(t,\xi)\;.\]
If $i$ is of a smaller index set, as with $\Aink$ for example, then the $\Aink$ is taken to be zero at the missing indices.
\end{definition}
With this auxiliary mapping $\mathfrak{P}$ at hand, we can define the sequences of functions we will deal with. 
\begin{definition}\label{def:IDF_approx_func}
Let $k,\tau_k$ and $\bxn{n}_k$ as in the preceeding section. Let furthermore $T=N_k\tau_k$. 
We define the approximate IDF $X_k:\TJ \rightarrow [a,b]$ as 
\[ X_k(t,\xi):= \mathfrak{P}[\bx](t,\xi) \;, \]
and the approximate spatial derivative $\dexi X_k:\TJ\rightarrow \R$ as 
\[ \dexi X_k(t,\xi):= \mathfrak{P}[\delta \bx](t,\xi) \;.\]

Furthermore let us define $\olAk ,\ulAk ,\Ak:\TJ\rightarrow \R$ as 
\[ \olAk(t,\xi):= \mathfrak{P}[\overline{\bA}](t,\xi)  \qquad  \ulAk(t,\xi):=\mathfrak{P}[\underline{\bA}](t,\xi) \qquad \text{ and } \qquad \Ak(t,\xi):= \mathfrak{P}[\bA](t,\xi) \]
where $\bA=\overline{\bA}+\underline{\bA}$ which consist of $\overline{A}_{i,k}^{(n)}$ and $\underline{A}_{i,k}^{(n)}$ given by (c.f. the r.h.s. of \eqref{Euler-Lagrange})
\[  \overline{{ A}}_{i,k}^{(n)}= \frac{\dhX\left( \delta \xn{n}_{i} \right)-\dhX\left( \delta \xn{n}_{i-1}\right)}{\dk}  \quad \text{ and } \quad  \underline{{ A}}_{i,k}^{(n)}=-\dvX(\xn{n}_i) \;. \]

Finally the r.h.s. of the discrete Euler-Lagrange equation in \textit{Lemma \ref{lem:Euler-Lag_disc}} will receive a corresponding sequence of functions called discrete velocity, in reference to the role its corresponding part plays in the PDE \eqref{PDE_u}. It will be denoted as $\CdelH_k:\TJ\rightarrow \R$ and it is defined as 
\[ \CdelH_k(t,\xi):=\dcs(\Ak(t,\xi)) \;.\]
\end{definition}

\begin{remarks}\label{rem:X_k_de_X_k} 
Before we proceed some remarks are in order. 
\begin{enumerate}
\item $X_k$ is uniformly bounded in $L^q(\TJ)$ for every $q\in [1,\infty]$. Indeed $\abs{X_k(t,\xi)}\leq \max\{\abs{a},\abs{b}\}$ by construction. 

The Hölder-type estimate from \textit{Lemma \ref{cor:Hoel_Est}} implies a similar estimate for our $X_k$: 
\[ \norm{X_k(t_1)-X_k(t_2)}_{L^1(\J)}\leq (t_1-t_2+\tau_k)^{1/p} \frac{1}{\sqrt[p]{\alpha}}\left( \HXd(\bxn{0})-\underline \HXd \right)^{1/p}  \]
holds for all $t_1,t_2\in \T$. 

Furthermore, for each $t$ $X_k(t,\cdot)$ it is a sequence of monotonously increasing functions and therefore converges, up to a subsequence pointwise $\Lc$-a.e. to some monotonously increasing function (c.f. \cite[Lemma 3.3.1]{GradientFlows}). By dominated convergence theorem, that subsequence also converges in $L^q(J)$ for every $q\in [1,\infty)$ to a limit $X_*(t,\cdot)$. 

Finally, combining the $L^1(\J)$-convergence up to a subsequence at every $t$ with the Hölder-continuity of $X_*(t):\T\rightarrow L^1(\J)$, resulting from the Hölder-type estimate above, we receive convergence, again up to a subsequence, of $X_k$ in $L^1(\TJ)$ and consequently in $L^q(\TJ)$ by a diagonal argument. 

\item The sequence $\dexi X_k$ is uniformly bounded in $L^q(\TJ)$ for every $q\in [1,\infty]$, too, by the maximum-/minimum principle.

Lacking the pointwise convergence, it will be the aim of the following subsection to establish convergence of $\dexi X_k$ in measure to allow us to apply the dominated convergence theorem to prove strong convergence of this sequence.
\end{enumerate}
\end{remarks}

\subsection{Strong convergence of $\dexi X_k$}\label{sec:str_conv_delX}

We will use the definition of the set of functions of bounded variation on an open set $\Omega$, $BV(\Omega)$ and of the total variation $\Var(f,\Omega)$ of such a function as was introduced in \cite{DeCiccoFusco}. 

\subsubsection{Tightness w.r.t. $\F$}
Consider the functional $\F:L^1(\J)\rightarrow [0,\infty]$ defined as
\begin{align}\label{def:F}
\F(Y):=\begin{cases}
\Var(Y,\Jo) &\text{ if } Y\in BV(\Jo) \\
+\infty  &\text{ elsewhere.}
\end{cases}
\end{align}

\begin{lemma}\label{lem:Ros_Sav_Func}
The functional $\F$ defined in \eqref{def:F} is normal and coercive in the sense of \cite[(1.7 a-c)]{RossiSavare}
\end{lemma}

\begin{proof}\ 
\begin{enumerate}
\item \textit{normal:} $\Var(\cdot, \Jo)$ is lower semicontinuous w.r.t. $L^1$-convergence \cite[Thm. 1.9]{Giusti}.
\item \textit{coercive:} Consider $A_c:=\F^{-1}((-\infty,c])$. By definition of $\F$, $A_c\subset BV(\Jo)$ and the $BV$-norm of elements of $A_c$ is uniformly bounded by $c$. Consequently, by \cite[Thm. 1.19]{Giusti}, $A_c$ is $L^1(\Jo)$-strongly compact. 
\end{enumerate}
\end{proof}

The above lemma makes sure that $\F$ is suitable. Now we show that $\dexi X_k$ is tight w.r.t. $\F$.

\begin{lemma}\label{lem:tight_wrt_F}
Our sequence $\dexi X_k$ is tight w.r.t. $\F$, that is to say 
\begin{align}
\sup_{k\in \N} \int_0^T \F (\dexi X_k(t,\cdot)) \de t < \infty\;.
\end{align}
\end{lemma}

\begin{proof}
We begin with 
\begin{align*}
\int_0^T \F(\dexi  X_k(t,\cdot)) \de t&=  \int_0^T \Var(\dexi  X_k(t,\cdot),\Jo) \de t \\
&= \tau_k \sum_{n=1}^{N_k} \Var(\dexi  X_k(n\tau_k,\cdot ),\Jo)\;.
\end{align*}
Now note that for every $t\in [0,T]$ we have that $X_k(t,\cdot):\J\rightarrow \R$ is a piecewise constant function and therefore its total variation can be calculated as the sum over the modulus of the jumps. This gives us
\begin{align*}
 \tau_k \sum_{n=1}^{N_k} \Var(\dexi  X_k(n\tau_k,\cdot ),\Jo)&=  \tau_k \sum_{n=1}^{N_k} \sum_{i=0}^{k-1} \abs{\delta \xn{n}_{i+1}-\delta \xn{n}_i } \\
 &=  \tau_k \sum_{n=1}^{N_k} \dk \sum_{i=0}^{k-1} \abs{\frac{\delta \xn{n}_{i+1}-\delta \xn{n}_i}{\dk} }\;.
\end{align*}
Now our goal is to utilize the estimates concerning $\Aink$ from \textit{Proposition \ref{prop:Entropy_ineq_conseq}}. To that end we see that by maximum-/minimum principle, positivity and monotony of $h_m^\pprime$ we receive a $D>0$ such that we have the lower bound 
\[ \abs{\delta \xn{n}_{i+1}-\delta \xn{n}_i }\geq D \abs{\dhX (\delta \xn{n}_{i+1} )-\dhX (\delta \xn{n}_{i} ) }\;.  \]
Incorporating this estimate we arrive at 
\begin{align*}
\int_0^T \F(\dexi  X_k(t,\cdot)) \de t
&\leq  D \tau_k \sum_{n=1}^{N_k} \dk \sum_{i=1}^{k-1} \abs{\Aink } + D \tau_k \sum_{n=1}^{N_k} \dk \sum_{i=0}^{k-1} \abs{ \dvX(\xn{n}_i )} \;.
\end{align*}
Now the first sums are bounded by \textit{Proposition \ref{prop:Entropy_ineq_conseq}} and the second sums are, by continuity of $\dvX$. 

This shows the sought for bound and therefore $\dexi X_k$ is tight w.r.t. $\F$. 
\end{proof}

\subsubsection{Equiintegrability w.r.t. the pairing $\g$.}

Now let $Y\in L^1(\J)$. Then we can define $I[Y](\xi):=a+\int_0^\xi Y(\zeta)\de \zeta$ where $I[Y]\in W^{1,1}(\Jo)$. 

Next we consider the functional $\g:L^p(\J)\times L^p(\J)\rightarrow [0,\infty)$ defined as 
\begin{align}\label{def:g_Savare}
\g(Y,Z):=\norm{I[Y]-I[Z]}_{L^1(\J)}\;.
\end{align}
Then the following lemma holds

\begin{lemma}\label{lem:Ros_Sav_Dist}
$\g$ is lower semicontinuous w.r.t. strong $L^p(\J)\times L^p(\J)$ topology.
\end{lemma}

\begin{proof}
Let $Y_k\xrightarrow{L^p(\J)} Y$ and let $\hat{Y}=I[Y_k]$. Then $\hat{Y}_k(0)=0$ and $\partial_\xi \hat{Y}_k=Y_k$. This already implies that $\hat{Y}_k\xrightarrow{L^p(\J)}\hat{Y}$ where $\hat{Y}\in W^{1,p}(\Jo)$ with $\partial_\xi \hat{Y}=Y$.

Now consider two sequences $Y_k,Z_k$ converging in $L^p(\J)$ to $Y$ and $Z$ respectively. Then $I[Y_k]-I[Z_k]=I[Y_k-Z_k]-a$ converges in $L^p(\J)$ to $I[Y-Z]-a$ and the $L^1(\J)$-norm is lower semicontinuous w.r.t. $L^p(\J)$ convergence. 
\end{proof}

Finally we show the equiintegrability w.r.t. $\g$. 
\begin{lemma}\label{lem:equiint_wrt_g}
The sequence $\dexi  X_k$ is equiintegrable w.r.t. $\g$, that is to say 
\begin{align}
\lim_{r\searrow 0} \sup_{k\in \N} \int_0^{T-r} \g\left(\dexi  X_k(t+r),\dexi  X_k(t)\right) \de t =0\;.
\end{align}
\end{lemma}

\begin{proof} 
First we establish an estimate concerning $I[\dexi  X_k]$ and $X_k$. Let $t\in [0,T-r]$ and $j(\xi)=\floor{\xi/d_k}$
\begin{align*}
&\norm{I[\dexi  X_k(t+r)]-I[\dexi  X_k(t)]}_{L^1(\J)}\\ 
&\qquad =\int_0^1 \abs{ \int_0^\xi \dexi  X_k(t+r,\zeta)-\dexi  X_k(t,\zeta)\de \zeta } \de \xi \\
&\qquad\leq \int_0^1 \abs{ X_k(t+r,\tfrac{j(\xi)}{k})-X_k(t,\tfrac{j(\xi)}{k})} \de \xi  \\
&\qquad \qquad + \frac{1}{\dk} \int_0^1 \int_{j(\xi)/k}^\xi \abs{X_k(t+r,\tfrac{j(\xi)+1}{k})-X_k(t+r,\tfrac{j(\xi)}{k}) -(X_k(t,\tfrac{j(\xi)+1}{k})-X_k(t,\tfrac{j(\xi)}{k}))   \de \zeta } \de \xi \\
&\qquad \leq \norm{X_k(t+r)-X_k(t)}_{L^1(J)}\\
&\qquad \qquad +\left( \sum_{j=0}^{k-1}  \abs{X_k(t+r,\tfrac{j+1}{k})-X_k(t,\tfrac{j+1}{k})} +\abs{(X_k(t+r,\tfrac{j}{k}) -X_k(t,\tfrac{j}{k}))}\right) \int_{jd_k}^{(j+1)d_k} \frac{\xi-\tfrac{j}{k}}{\dk} \de \zeta \\
&\leq 3 \norm{X_k(t+r)-X_k(t)}_{L^1(\J)}\;.
\end{align*}
Together with the H\"older-estimate from \textit{Corollary \ref{cor:Hoel_Est}}, this yields, for some constant $C$
\begin{align*}
\lim_{r\searrow 0} \sup_{k\in \N} \int_0^{T-r} \g\left(\dexi  X_k(t+r),\dexi  X_k(t)\right) \de t 
&\leq 3\lim_{r\searrow 0} \sup_{k\in \N} \int_0^{T-r}   r^{1/p}C \de t 
=0\;. \qedhere
\end{align*}
\end{proof}

\begin{proposition}
The sequence $\dexi  X_k$ is sequentially compact w.r.t. strong convergence in $L^{p^\prime}(\T, L^{p^\prime}(\J))$. 
\end{proposition}

\begin{proof}
This will be an application of \cite[Thm. 2]{RossiSavare}. Note that $t\mapsto \dexi  X_k(t)$ is a sequence of $L^1(\J)$-valued functions. The functional $\F$ is normal and coercive in the sense of \cite[(1.7 a-c)]{RossiSavare} as was shown in \textit{Lemma \ref{lem:Ros_Sav_Func}} and \textit{Lemma \ref{lem:Ros_Sav_Dist}} shows that the substitute distance $\g$ is lower semicontinuous. Additionally, our sequence is tight w.r.t. $\F$ and it is equiintegrable w.r.t. $\g$ which was shown in \textit{Lemma \ref{lem:tight_wrt_F}} and \textit{Lemma \ref{lem:equiint_wrt_g}} respectively. 

We still have to show the compatibility of $\g$ with the functionals $\F$. So let $Y,Z\in L^1(\J)$ with $\F(Y),\F(Z)<\infty$. Then $\norm{I[Y]-I[Z]}_{L^1(\J)}=0$ implies that $I[Y]=I[Z]$ a.e. and consequently $Y=Z$.

This shows that Theorem 2 of \cite{RossiSavare} is applicable and we receive a subsequence of $\dexi  X_k$ that converges in measure as a curve $\dexi X_k:\T\rightarrow L^1(\J)$ to some limit curve $Y$. 

Furthermore, since $\dexi  X_k $ is dominated on $\T$ by the maximum-/minimum principle, we can enhance this result, possibly by passing to a subsubsequence, to strong convergence in $L^{p^\prime}(\T, L^{p^\prime}(\J))$.
\end{proof}

Finally we have to establish a connection between the cluster points of $\dexi X_k$ and $X_\ast$ the limit of $X_k$.

\begin{corollary}\label{cor:ident_dexi_X_*}
Up to a subsequence $\dexi X_k$ converges strongly w.r.t. $L^{p^\prime}([0,T]\times(0,1))$ to $\partial_\xi X_\ast$, $\hX(\dexi  X_k)$ to $ \hX(\partial_\xi X_\ast)$ and  $\dhX(\dexi  X_k)$ to $ \dhX(\partial_\xi X_\ast)$.
\end{corollary}

\begin{proof}
The first convergence follows by by \textit{Appendix \ref{app:diff_quot} } and the last two by the first convergence and \textit{Appendix \ref{app:str_conv} }.
\end{proof}

\subsection{The limit of the Euler-Lagrange equation}

The discrete Euler-Lagrange equation we received in \textit{Lemma \ref{lem:Euler-Lag_disc}} expressed in terms of approximate IDFs takes the form 
\begin{align}\label{eq:Eul-Lag_whole_IDF}
 \detau X_k= \CdelH_k
\end{align}
where we abbreviated the temporal backwards difference quotient $\detau X_k:\TJoc\rightarrow \R$ of $X_k$ by  
\[ \detau X_k(t,\xi) := \frac{X_k(t,\xi)-X_k(t-\tau_k,\xi)}{\tau_k}=\sum_{n=1}^{N_k} \sum_{i=1}^{k} \frac{\xn{n}_i-\xn{n-1}_i}{\tau_k} \IndTX(t,\xi)\;. \]

We want to prove a convergence result for the sequence $\CdelH_k$ next, by transferring the bounds found in \textit{Proposition \ref{prop:Entropy_ineq_conseq}} to our newly defined sequences. 

\begin{lemma}\label{lem:CdelH_bounded}
The sequences $\CdelH_k$, and consequently $\detau X_k$, converge, up to a subsequence, weakly in $L^p(\TJ)$ to the limit $\CdelHs\in L^p(\TJ)$. 
\end{lemma}

\begin{proof}
The result for $\detau X_k$ follows by \eqref{eq:Eul-Lag_whole_IDF}, so we only have to show the claim for $\CdelH_k$.

We receive the uniform bound on the $L^p$-norm by \textit{Proposition \ref{prop:Entropy_ineq_conseq} } as
\begin{align*}
\norm{\CdelH_k}_{L^p(\TJ)}^p
&= \tau_k\sum_{n=0}^{N_k}  \dk \sum_{i=1}^{k-1}  \abs{\dcs\left(\Aink\right)}^p \leq 2 \frac{1}{\alpha} \left(\sup_{k}\HXd(\bxn{0}_k)- \underline{\HXd}\right)\;.
\end{align*}
 
This shows that the sequence is uniformly bounded w.r.t. $L^p(\TJ)$-norm which in turn implies, by Banach-Alaoglu, weak$^*$-compactness and by equivalence of weak and weak$^*$ convergence in $L^p(\TJ)$ we receive the claim.
\end{proof}


\begin{lemma}\label{lem:A_k-Conv_p-Wass}
In the $p$-Wasserstein case, up to a subsequence, $\Ak$ converges weakly in $L^{p^\prime}(\TJ)$, to $\partial_\xi \dhX(\partial_\xi X_*)-\dvX(X_*)$ which we will abbreviate consistently as 
\begin{align*}
A_*:=\partial_\xi \dhX(\partial_\xi X_*)-\dvX(X_*)\;.
\end{align*} 
\end{lemma}

\begin{proof}
Using the bound on $\Ak$ from \textit{Proposition \ref{prop:Entropy_ineq_conseq}} we receive
\begin{align*}
\norm{\Ak}_{L^{p^\prime}(\TJ)}^{p^\prime} 
&= \tau_k \sum_{n=1}^{N_k} \dk \sum_{i=1}^{k-1} \abs{\Aink}^{p^\prime} \leq 2\sqrt[p-1]{\beta} \left(\sup_{k}\HXd(\bxn{0}_k)- \underline{\HXd}\right)\;.
\end{align*}
With this bound at hand and Banach-Alaoglu we arrive at a weakly convergent subsequence (w.r.t. $L^{p^\prime}(\TJ)$). 

In order to identify the limit, we use $\Ak=\olAk+\ulAk$ from \textit{Definition \ref{def:IDF_approx_func}}. 

$\ulAk$ is the Lipschitz-function $\dvX$ applied to $X_k$ and with \emph{Appendix \ref{app:str_conv} } we receive strong convergence in $L^{p^\prime}(\TJ)$ implying weak convergence. 

Now 
$$\Ak - \olAk=\ulAk= \frac{\dhX(\dexi X_k(t,\xi+\dk))-\dhX(\dexi X_k(t,\xi))}{\dk}=\dexi \dhX(\dexi X_k(t,\xi)) $$ 
converges weakly w.r.t. $L^{p^\prime}(\TJ)$ and is consequently bounded therein. The claim now follows from \emph{Appendix \ref{app:diff_quot} }.
%
%
\end{proof}

To finally receive the limit of the Euler-Lagrange equation \eqref{eq:Eul-Lag_whole_IDF}, we identify the limit of $\detau X_k$ with the weak derivative w.r.t. $t$ of the limit $X_*$. This will then give us the equation 
\[ \partial_t X_* = \CdelHs\;. \] 

Indeed, we can apply \emph{Appendix \ref{app:diff_quot} } to it and receive the limit of our Euler-Lagrange equation
\begin{align}\label{eq:Eul-Lag_Limit}
\partial_t X_* = \CdelHs
\end{align}
which holds on $\TJoo$.

\section{Identification of the nonlinear limit}
In the preceeding section we have established some convergence results for $X_k$, $\dexi X_k$, $\Ak$ and $\CdelH_k$, some only up to subsequences. In this section we assume that $X_k$ etc. are already non-relabelled subsequences, such that all of the above convergence results hold. 

So far we have shown that our Euler-Lagrange equation admits a limit, but the limit of the r.h.s. is, by nonlinearity of $\dcs$ still not identified. To identify this limit and therefore receive our IDF $X_\ast$ as a weak solution of \eqref{PDE_u} in terms of IDF we will make use of a Browder-Minty argument w.r.t. the monotone $\dcs$. 

%

The monotonicity inequality we want to consider will be 
\begin{align}
0 \leq (\dcs ( \Ak(t,\xi) ) - \dcs (\As(t,\xi)-\varepsilon \psi(\xi) ))(\Ak(t,\xi) - (\As(t,\xi)-\varepsilon \psi(\xi)))
\end{align}
where $\varepsilon>0$ and $\psi \in C^\infty(\J)$. 
Since the above inequality holds for every $(t,\xi)\in \TJoo$ we can integrate it weighted in time with some non-negative $u\in C_c^\infty((0,T))$ and making use of the abbreviation $\CdelH_k$ to arrive at 
\begin{align}\label{ineq:Brow_Min_monot}
0\leq \int_0^T \int_0^1 (\CdelH_k(t,\xi) - \dcs (\As(t,\xi)-\varepsilon \psi(\xi) ))(\Ak(t,\xi) - (\As(t,\xi)-\varepsilon \psi(\xi))) u(t) \de \xi \de t \;.
\end{align}

In order to prove the limit in \eqref{ineq:Brow_Min_monot}, we will split the expression up by its bilinearity. Plugging together the results and using subadditivity of the $\limsup$ then shows the sought for estimate
\begin{align}\label{ineq:Browder-Minty}
\begin{split}
\limsup_{k\rightarrow \infty} &\int_0^T \int_0^1 (\CdelH_k(t,\xi) - \dcs (\As(t,\xi)-\varepsilon \psi(\xi) ))(\Ak(t,\xi) - (\As(t,\xi)-\varepsilon \psi(\xi))) u(t) \de \xi \de t \\ 
&\leq \int_0^T \int_0^1 [\CdelHs(t,\xi) - \dcs (\As(t,\xi)-\varepsilon \psi(\xi) )]\;\varepsilon \psi(\xi) u(t) \de \xi \de t \;.
\end{split}
\end{align}

\subsection{ }\label{sec:id_the_limit_I}

We begin with splitting up $\Ak=\olAk-\ulAk$ 
\begin{align*}
\int_0^T \int_0^1& \CdelH_k(t,\xi) \Ak(t,\xi)  u(t) \de \xi \de t \\
&= \int_0^T \int_0^1 \CdelH_k(t,\xi) \olAk(t,\xi)  u(t) \de \xi \de t-\int_0^T \int_0^1 \CdelH_k(t,\xi) \ulAk(t,\xi)  u(t) \de \xi \de t\;.
\end{align*}
The second integral can be treated easily, since $\ulAk\rightarrow \dvX(X_*(t,\xi))$ strongly w.r.t. $L^{p^\prime}(\TJ)$ by \textit{Lemma \ref{lem:A_k-Conv_p-Wass}} and with \textit{Lemma \ref{lem:CdelH_bounded}} we receive
\begin{align*}
\lim_{k\rightarrow \infty} \int_0^T \int_0^1 \CdelH_k(t,\xi) \ulAk(t,\xi)  u(t) \de \xi \de t = \int_0^T \int_0^1 \CdelH_*(t,\xi) \dvX(X_*(t,\xi))  u(t) \de \xi \de t
\end{align*}
right away.

To deal with the remaining integral will require more work. We begin by applying the Euler-Lagrange equation \eqref{eq:Eul-Lag_whole_IDF} to the integral to receive 
\begin{align*}
\int_0^T \int_0^1 \CdelH_k(t,\xi) \olAk(t,\xi)  u(t) \de \xi \de t &= \int_0^T \int_0^1 \detau X_k(t,\xi) \olAk(t,\xi)  u(t) \de \xi \de t\;.
\end{align*}

We begin with exploiting the piecewise constant structure
\begin{align*}
\int_0^T \int_0^1& \detau X_k(t,\xi) \olAk(t,\xi)  u(t) \de \xi \de t \\ 
&= \int_0^T \int_0^1 \sum_{n=1}^{N_k} \sum_{i=1}^{k-1} \frac{\xn{n}_i - \xn{n-1}_i}{\tau_k} \frac{\dhX(\delta \xn{n}_i) - \dhX(\delta \xn{n}_{i-1})}{\dk} \IndTX(t,\xi)  u(t) \de \xi \de t \\
&= \dk \sum_{n=1}^{N_k} \sum_{i=1}^{k-1} \frac{\xn{n}_i - \xn{n-1}_i}{\tau_k} \frac{\dhX(\delta \xn{n}_i) - \dhX(\delta \xn{n}_{i-1})}{\dk} \int_{(n-1)\tau_k}^{n\tau_k}  u(t)  \de t \;.
\end{align*}
We will abbreviate the integral in the last line as  $Z_n^u:=\int_{(n-1)\tau_k}^{n\tau_k}  u(t)  \de t$.

Now we apply summation by parts w.r.t. $i$, a convexity estimate and summation by parts w.r.t. $n$ to receive 
\begin{align*}
\sum_{n=1}^{N_k} \sum_{i=1}^{k-1}& \frac{\dhX(\delta \xn{n}_i) - \dhX(\delta \xn{n}_{i-1})}{\dk} \frac{\xn{n}_i - \xn{n-1}_i}{\tau_k}  \dk  Z_n^u
&\leq -\sum_{n=1}^{N_k} \sum_{i=1}^{k-1}   \frac{h_m(\delta \xn{n}_i) - h_m(\delta \xn{n-1}_i)}{\tau_k}  \dk  Z_n^u \\
&=\int_0^T \int_0^1   \hX(X_k(t,\xi))  \frac{u(t+\tau_k)-u(t)}{\tau_k} \de \xi \de t
\end{align*}
where the inequality is justified by the convexity of $h_m$. 

Now as was shown, $\hX(X_k)$ converges strongly to $h_m(\partial_\xi X_*)$ and since $u\in C_c^\infty((0,T))$ the difference quotient converges uniformly in $k$. Consequently we receive in the limit
\begin{align*}
\lim_{k\rightarrow \infty}\int_0^T \int_0^1   \hX(X_k(t,\xi))  \frac{u(t+\tau_k)-u(t)}{\tau_k} \de \xi \de t &=\int_0^T \int_0^1 \hX(\dexi X_*(t,\xi)) \partial_t u(t) \de \xi \de t\;.
\end{align*}
Now we would want to get the partial differentiation w.r.t. $t$ back to the $h_m$ again to receive by chain rule a possibility to get our $\CdelH_*$ back. Unfortunately the expression $h_m(\dexi X_*)$ does not have enough regularity to allow for that. But we can help ourselves by undoing the limit of just the difference quotient again.
\begin{align*}
\int_0^T \int_0^1 \hX(\dexi X_*(t,\xi))& \partial_t u(t) \de \xi \de t 
=\lim_{k\rightarrow \infty} \int_0^T \int_0^1 \hX(\dexi X_*(t,\xi)) \frac{u(t)- u(t-\tau_k)}{\tau_k} \de \xi \de t \\
&= \lim_{k\rightarrow \infty} -\int_0^T \int_0^1 \frac{\hX(\dexi X_*(t+\tau_k,\xi))-\hX(\dexi X_*(t,\xi))}{\tau_k} u(t) \de \xi \de t \\
&\leq  \lim_{k\rightarrow \infty} -\int_0^T \int_0^1 \dhX (\dexi X_*(t,\xi))\frac{\dexi X_*(t+\tau_k,\xi)-\dexi X_*(t,\xi)}{\tau_k} u(t) \de \xi \de t \\
&=\lim_{k\rightarrow \infty} \int_0^T \int_0^1 \partial_\xi \dhX (\dexi X_*(t,\xi))\frac{X_*(t+\tau_k,\xi)-X_*(t,\xi)}{\tau_k} u(t) \de \xi \de t  \\
&=\int_0^T \int_0^1 \partial_\xi \dhX (\dexi X_*(t,\xi))\CdelH_*(t,\xi) u(t) \de \xi \de t \;.
\end{align*}
where we used summation by parts w.r.t. $n$ disguised in $t+\tau_k$ and convexity of $h_m$ as well as integrate by parts. In the last step we used  $\partial_\xi \dhX (\dexi X_*(t,\xi)) u(t)\in L^{p^\prime}(\TJ)$ and $X_*\in W^{1,p}(\TJoo)$ which implies a strict enough convergence of the difference quotients to pass to the limit. 


\subsection{ }
The inequality we want to show next is 
\begin{align*}
\liminf_{k\rightarrow \infty} \int_0^T \int_0^1& \dcs\big(A(t,\xi) -\varepsilon \psi(\xi)\big) (\Ak(t,\xi) -(A(t,\xi) - \varepsilon \psi(\xi))) \de \xi \de t \\
&\geq \int_0^T \int_0^1 \dcs\big(A(t,\xi) -\varepsilon \psi(\xi)\big)  \varepsilon \psi(\xi) \de \xi \de t\;.
\end{align*}

Recall $\Ak$ converges w.r.t. $L^{p^\prime}(\TJ)$ weakly to $\As$ (\textit{Lemma \ref{lem:A_k-Conv_p-Wass}}). So we have established the sought for inequality as equation for the limit as soon as we can show $\dcs\big(\As -\varepsilon \psi\big)\in L^p(\TJ)$. 

Indeed we have by $\frac{1}{\sqrt[p-1]{\alpha}}\abs{r}^{\frac{1}{p-1}}\geq \abs{\dcs(r)}$ (c.f. proof of \textit{Proposition \ref{prop:Entropy_ineq_conseq}}) 
\begin{align*}
\int_0^T\int_0^1 \abs{\dcs\big(\As(t,\xi) -\varepsilon \psi(\xi)\big)}^p \de \xi \de t &\leq \frac{1}{\sqrt[p-1]{\alpha}} \norm{\As+\varepsilon \psi}_{L^{p^\prime}(\TJ)}^{p^\prime} \;.
\end{align*}
Now since $\psi$ is a test-function, $\norm{\As-\varepsilon \psi}_{L^{p^\prime}(\TJ)}^{p^\prime} <\infty $ and consequently $\dcs\big(\As -\varepsilon \psi\big)\in L^p(\TJ)$.

\subsection{ }
The final inequality we will show to identify the limit $\CdelH_*$ of $\CdelH_k$ is 
\begin{align*}
\liminf_{k\rightarrow \infty} \int_0^T \int_0^1& \CdelH_k(t,\xi) (\As(t,\xi)-\varepsilon\psi(\xi)) \de \xi \de t\geq \int_0^T \int_0^1\CdelHs(t,\xi) (\As(t,\xi)-\varepsilon\psi(\xi)) \de \xi \de t \;.
\end{align*}
Recall \textit{Lemma \ref{lem:CdelH_bounded}} which states that $\CdelH_k\rightarrow \CdelHs$ weakly in $L^p(\TJ)$. Again we receive the sought for limit inferior estimate as an equation for the limit since we already confirmed that $\As-\varepsilon\psi\in L^{p^\prime}(\TJ)$.

\subsection{Identification of the limit $\CdelHs$}
Plugging together the above calculations we arrive at \eqref{ineq:Browder-Minty}. Combining this with \eqref{ineq:Brow_Min_monot} we receive 
\begin{align*}
0 \leq \int_0^T \int_0^1 [\CdelHs(t,\xi) - \dcs (\As(t,\xi)-\varepsilon \psi(\xi) )]\;\varepsilon \psi(\xi) u(t) \de \xi \de t
\end{align*}
which holds for every $\varepsilon>0$, $u\in C_c^\infty((0,T))$ with $u\geq 0$ and $\psi\in C_c^\infty((0,1))$. Dividing by $\varepsilon>0$ and exchanging $\psi \leftrightarrow -\psi$ we receive the equation 
\begin{align*}
0 = \int_0^T \int_0^1 [\CdelHs(t,\xi) - \dcs (\As(t,\xi)-\varepsilon \psi(\xi) )] \psi(\xi) u(t) \de \xi \de t\;.
\end{align*}
Finally we send $\varepsilon\searrow 0$ and receive, since $\varepsilon \psi$ converges uniformly to zero for $\varepsilon\searrow 0$, 
\begin{align*}
0 = \int_0^T \int_0^1 [\CdelHs(t,\xi) - \dcs (\As(t,\xi))] \psi(\xi) u(t) \de \xi \de t
\end{align*}
which then implies the weak formulation of \eqref{eq:PDE_in_X} in \emph{Theorem \ref{Thm:Main_Thm} } by \eqref{eq:Eul-Lag_Limit}.

\section{Solution of our PDE in the sense of IDFs and the flux-limiting case}
\subsection{Solution of our PDE}
We will now show that the limit $X_*$ is indeed a solution to our PDE \eqref{PDE_u} in terms of IDF. To that end, we still have to show that $X_*(t)$ is indeed an inverse distribution function for every $t\in [0,T]$, i.e. it is non-decreasing and $X_*(t,0)=a$ as well as $X_*(t,1)=b$ holds for all $t\in \T$. 

Consider the sequence $X_k(t)$.We know $X_k(t)$ is non-decreasing and $X_k(t,0)=a$, $X_k(t,1)=b$ holds. These properties survive the pointwise limit, so $X_*(t)$ is again non-decreasing and $X_*(t,0)=a$ as well as $X_*(t,1)=b$ holds.

Furthermore we have to show that the initial data are assumed $X_*(t)\rightarrow X^{(0)}_*$ for $t\rightarrow 0$. This will be a consequence of $X_*$ being $1/p$-H\"older continuous as a curve in $L^1(\J)$ which, in turn, will be a consequence of \textit{Corollary \ref{cor:Hoel_Est} }. Indeed, we have for every $k$ and $t_1,t_2\in \T$ where $t_1>t_2$
\begin{align*}
\norm{X_k(t_1)-X_k(t_2)}_{L^1(\J)}&= \dk \sum_{i=0}^{k} \abs{\xn{m_2}_i - \xn{m_1}_i} \leq C (t_1-t_2)^{1/p}
\end{align*}
where $C$ does not depend on $k$, since we assumed $\sup_{k}\HXd(\bxn{0}_k)$ to be finite. The strong convergence $X_k(t)\rightarrow X_*(t)$ in $L^1(\J)$ established in \textit{Remark \ref{rem:X_k_de_X_k}} now yields the H\"older-continuity of $X_*$ and this implies in particular $\lim_{t\searrow 0} X_*(t)=X^{(0)}$ in the $L^1(\J)$ sense.

\subsection{The flux-limiting case}
To prove \textit{Theorem \ref{Thm:Scnd_Thm}}, we will show that, given the prerequisites of \textit{Theorem \ref{Thm:Scnd_Thm}}, the flux-limiting cost functions can be regularized to be actually $p$-Wasserstein cost without changing the minimizers of our algorithm steps. 

Let us assume $\c$, $\bxn{0}_k$ are as in \textit{Lemma \ref{lem:min_max_2}}. Then, by the very same lemma, the bounds for $\delta^2 \bxn{0}_k$ hold for all $\delta^2 \bxn{n}_k$, too. As a first consequence, this yields finite bounds from above and below for $\Aink$. Indeed, by the properties of $\hu$ and consequently $\hX$ we receive $ \ddxlb\cdot \ddhX\left(\dxub\right)\leq \Aink\leq  \ddxub\cdot \ddhX\left(\dxlb\right) $ if $\ddxlb<0<\ddxub$. If it is the case that $\ddxlb$ and $\ddxub$ lie on the same side of zero, one of the bounds in the interval has to be replaced with zero. 

Now since $\dcs$ is monotonously increasing, we receive bounds for the discrete temporal backward difference 
\[ \frac{\xn{n}_i-\xn{n-1}_i}{\tau_k}= \dcs(\Aink)\in \left[\dcs\left(\ddhX(\dxub)\ddxlb\right),\dcs\left( \ddhX(\dxlb)\ddxub \right)\right] \]
again with the appropriate corrections if $\ddxlb$ and $\ddxub$ lie on the same side of zero. So to summarize, there are uniform bounds $ \underline{C}>-\ls$ and $\overline{C}<\ls $ such that for every $k,n$ our minimization problem in the algorithm can be narrowed down to a minimization over $\bx$ such that $ \frac{\xn{n}_i-\xn{n-1}_i}{\tau_k}\in [\underline{C},\overline{C}]$. This allows us in particular to regularize the flux-limiting cost $\c$ outside of $[\frac{\underline{C}-\ls}{2},\frac{\overline{C}+\ls}{2}]$ to satisfy \aspa\ and \aspb\ (\as\ holds already by assumption) showing that the results concerning $p$-Wasserstein-like soct hold for flux-limiting cost as well, as soon as we have initial data that satisfies the additional regularity assumption from \textit{Theorem \ref{Thm:Scnd_Thm}}.

\section{Numerical experiments}

\subsection{Implementation}

We perform the minimization of $F:\bx\mapsto \Phi(\tau;\bx^{(n-1)},\bx)$ in Algorithm \eqref{Minim_X} by a damped Newton scheme
\begin{align*}
p_j & = -HF(\bx_j)^{-1}\nabla F(\bx_j) \\
\bx_{j+1} & = \bx_j + h_j p_j, \quad j=0,1,\ldots
\end{align*}
for the gradient of $F$.  The choice of the step size $h_j$ in each step is realized by an Armijo-type heuristic, i.e.\ we choose $h_j$ as the largest value from the sequence $\{1,\frac12,\frac14,\ldots\}$ for which
\begin{align*}
\bx_j+h_j p_j & \in \{\bx\in \R^{k+1}\mid a=x_0<x_1<\dots<x_k=b, 
| x_i - x_i^{(n-1)}| \leq \tau\},
\end{align*}
i.e.\ such that the next iterate $\bx_{j+1}$ is still an IDF and has a well defined optimal transport distance in the flux limited case.

A Matlab code for the following experiments is given in the Appendix.

\subsection{Linear diffusion}

We start with the case $m=1$, i.e.\ the case of the Boltzmann entropy for the internal energy potential. All experiments have been carried out with $k=1000$ grid points and time step $\tau=0.01$. Figs.~\ref{fig:evol_p=7} and \ref{fig:evol_p=7_left} show the evolution of an initial distribution with localized support over the time interval $[0,2]$ for Wasserstein costs with $p=7$, while Fig.~\ref{fig:evol_rel} shows the same evolution for the flux limited case with $c$ given by \eqref{eq:c_rel}, $\gamma=1$.

In Fig.~\ref{fig:convergence} (left), we depict the $L_1$-error of the computed density in dependence of the mesh size.  The error is estimated by computing the exact $L_1$ difference to a reference solution on a grid with 10000 points, the same initial condition as for Fig.~\ref{fig:evol_p=7} has been used.  The experiments suggest that the error decreases linearly with the grid size. To the right in this figure, the $L_1$-error of the computed density (on a grid of 1000 points) in dependence of the time step is plotted.  Again, we estimate this error by comparing to a reference solution, here with time step $\tau=0.001$.  The result clearly suggests a linear dependence of this error on the time step. 

\begin{figure}[h]
\centering
\includegraphics[width=0.42\textwidth]{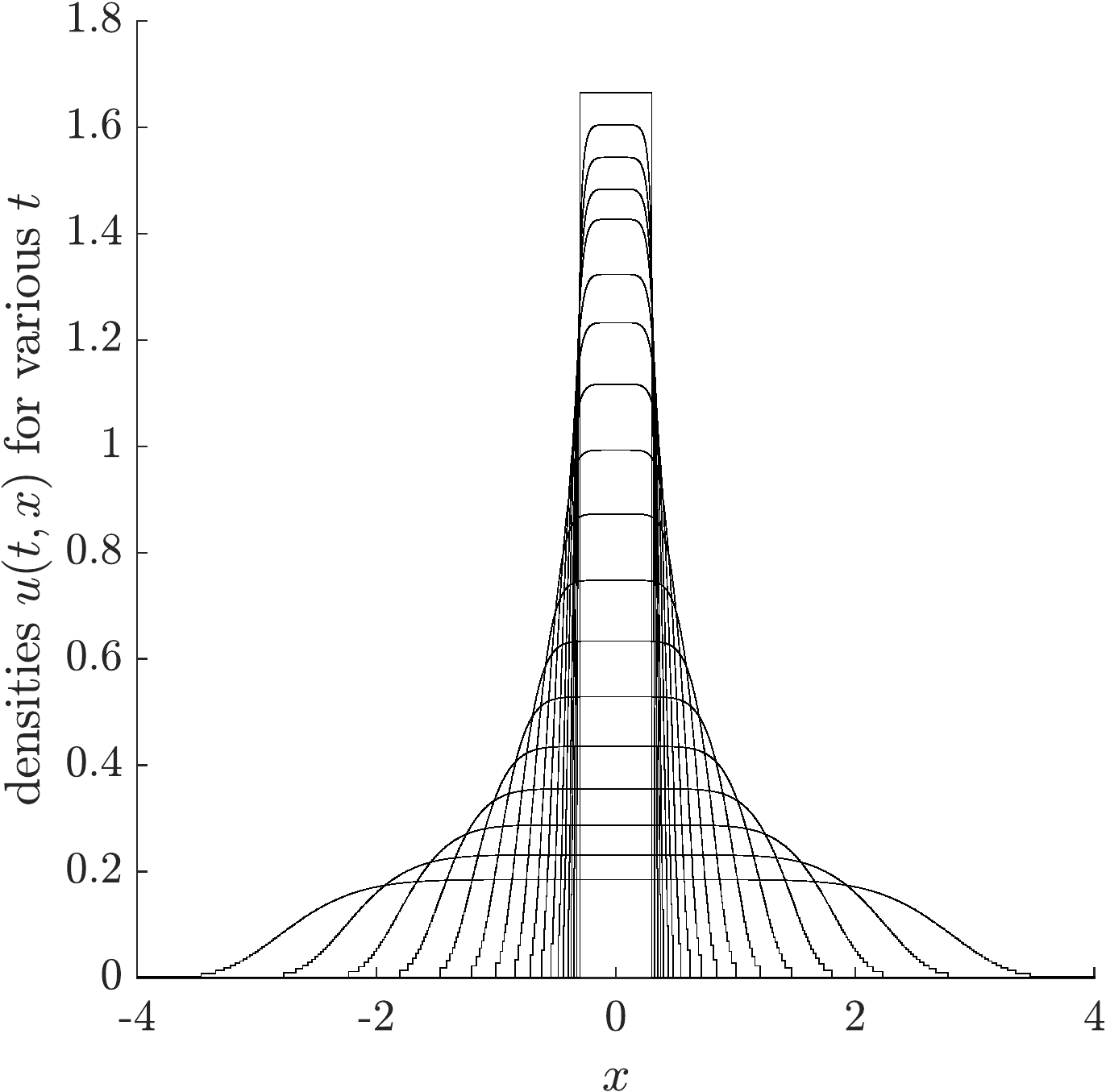}
\includegraphics[width=0.42\textwidth]{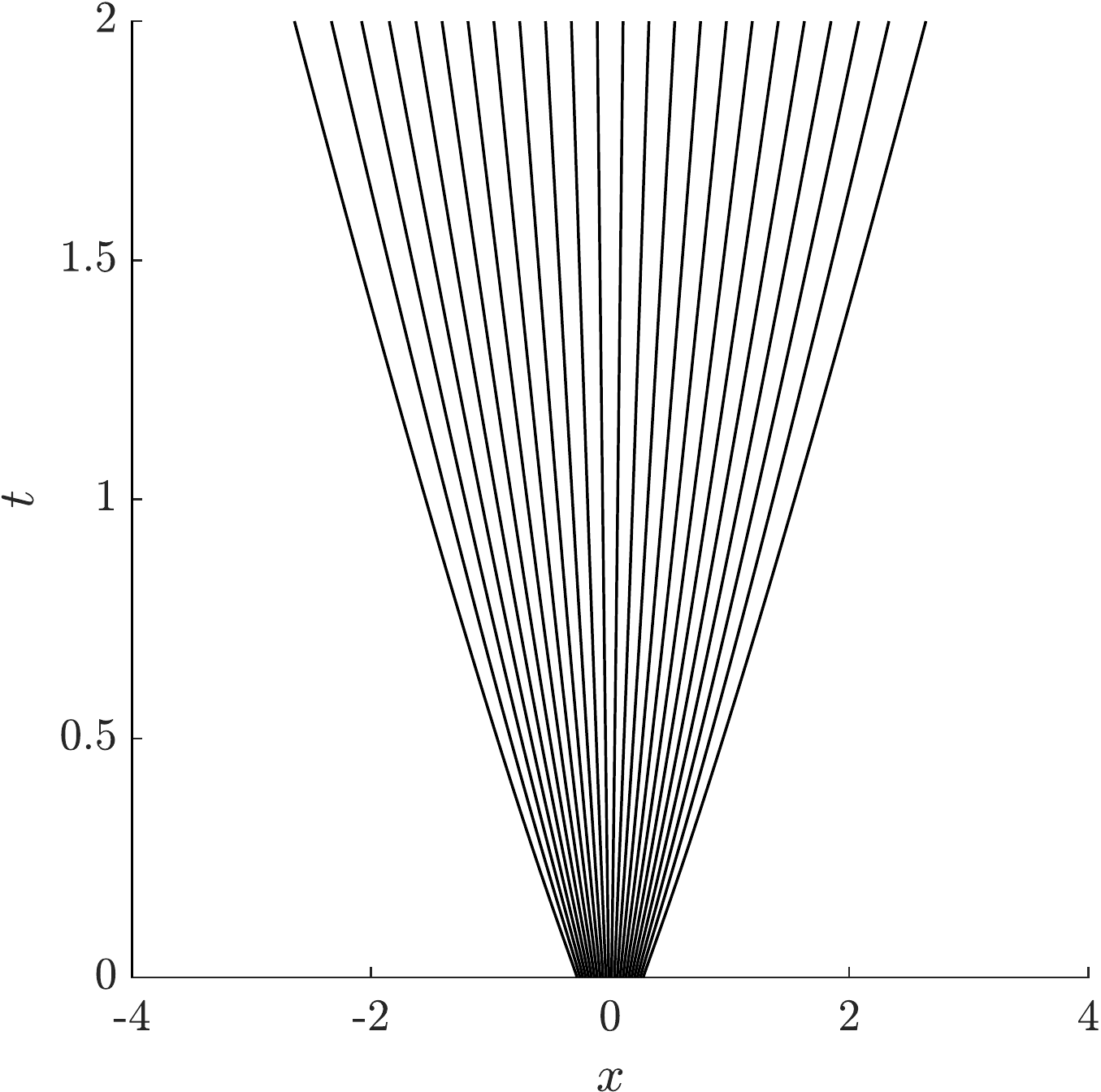}
\caption{Experiment: p-Wasserstein cost, linear diffusion. Left: Approximate densities $u(t,\cdot)$ at $t=0.01\cdot 10^k$, $k=0,0.12,0.24,\ldots,\log_{10}(200)$, initial mass uniformly distributed on $[-0.3,0.3]$. Right: the corresponding characteristics.}
\label{fig:evol_p=7}
\end{figure}

\begin{figure}[h]
\centering
\includegraphics[width=0.42\textwidth]{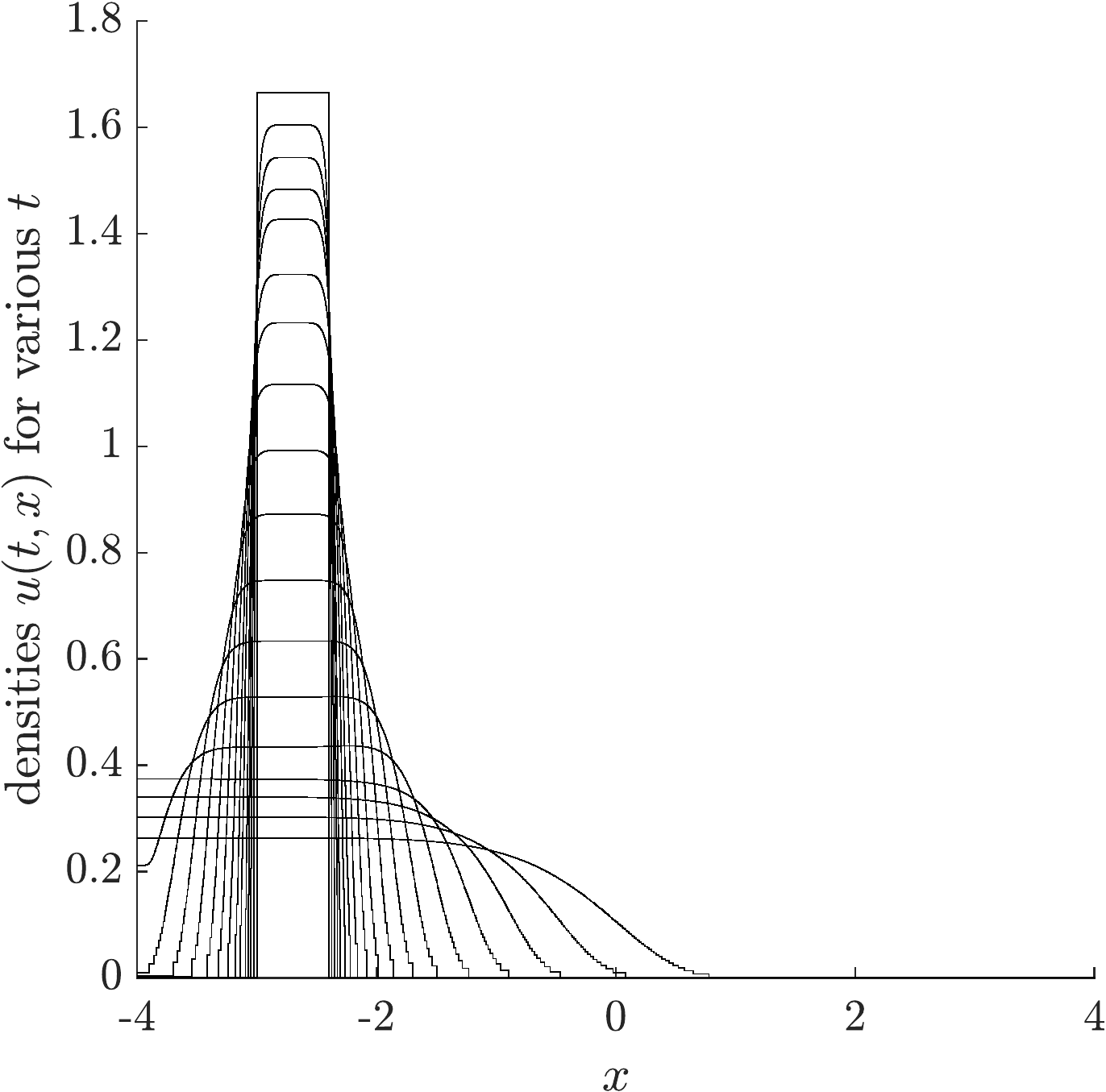}
\includegraphics[width=0.42\textwidth]{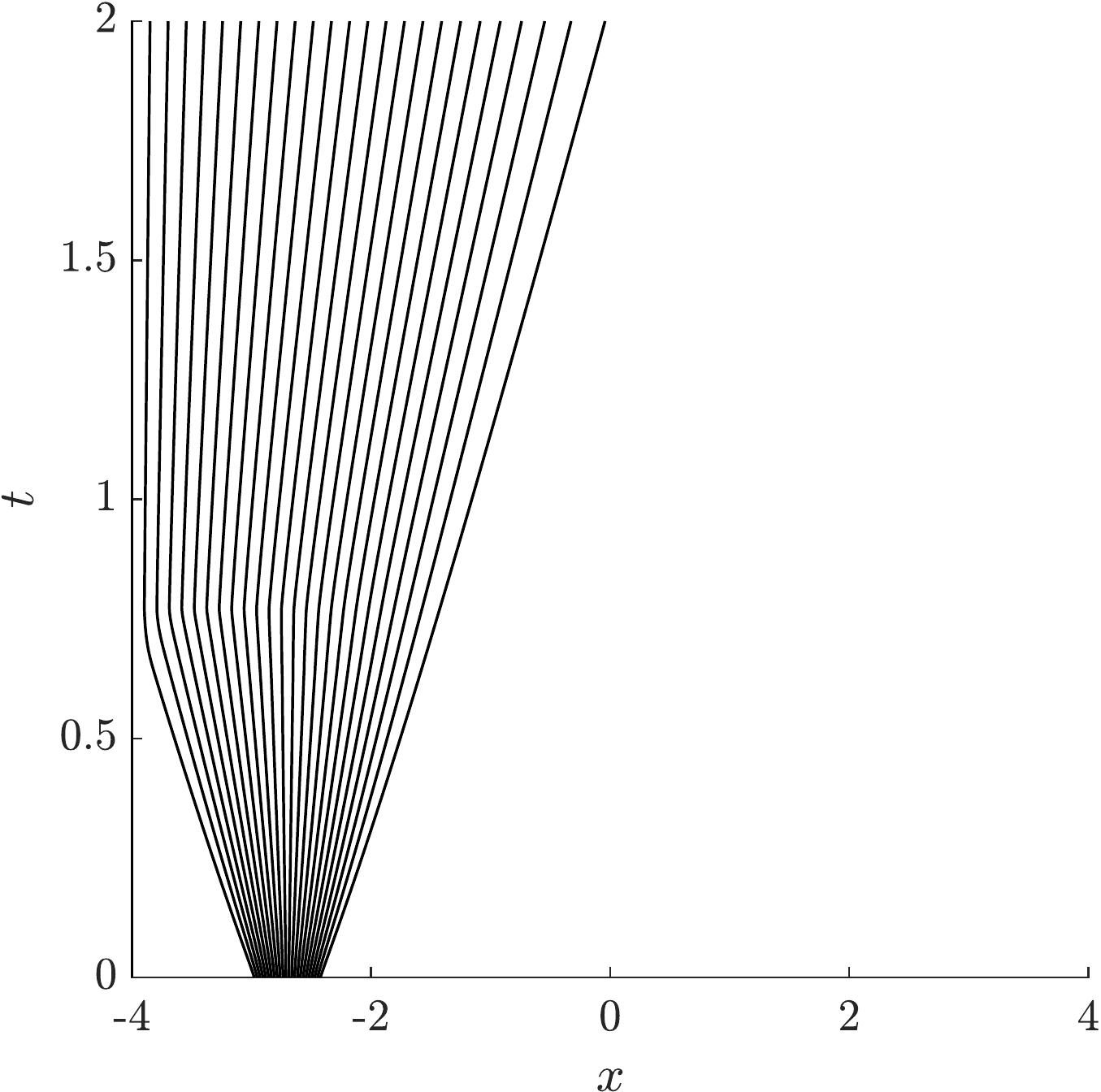}
\caption{Experiment: p-Wasserstein cost, linear diffusion. Left: Approximate densities $u(t,\cdot)$ at $t=0.01\cdot 10^k$, $k=0,0.12,0.24,\ldots,\log_{10}(200)$, initial mass uniformly distributed on $[-3,-2.4]$. Right: the corresponding characteristics.}
\label{fig:evol_p=7_left}
\end{figure}

\begin{figure}[h]
\centering
\includegraphics[width=0.42\textwidth]{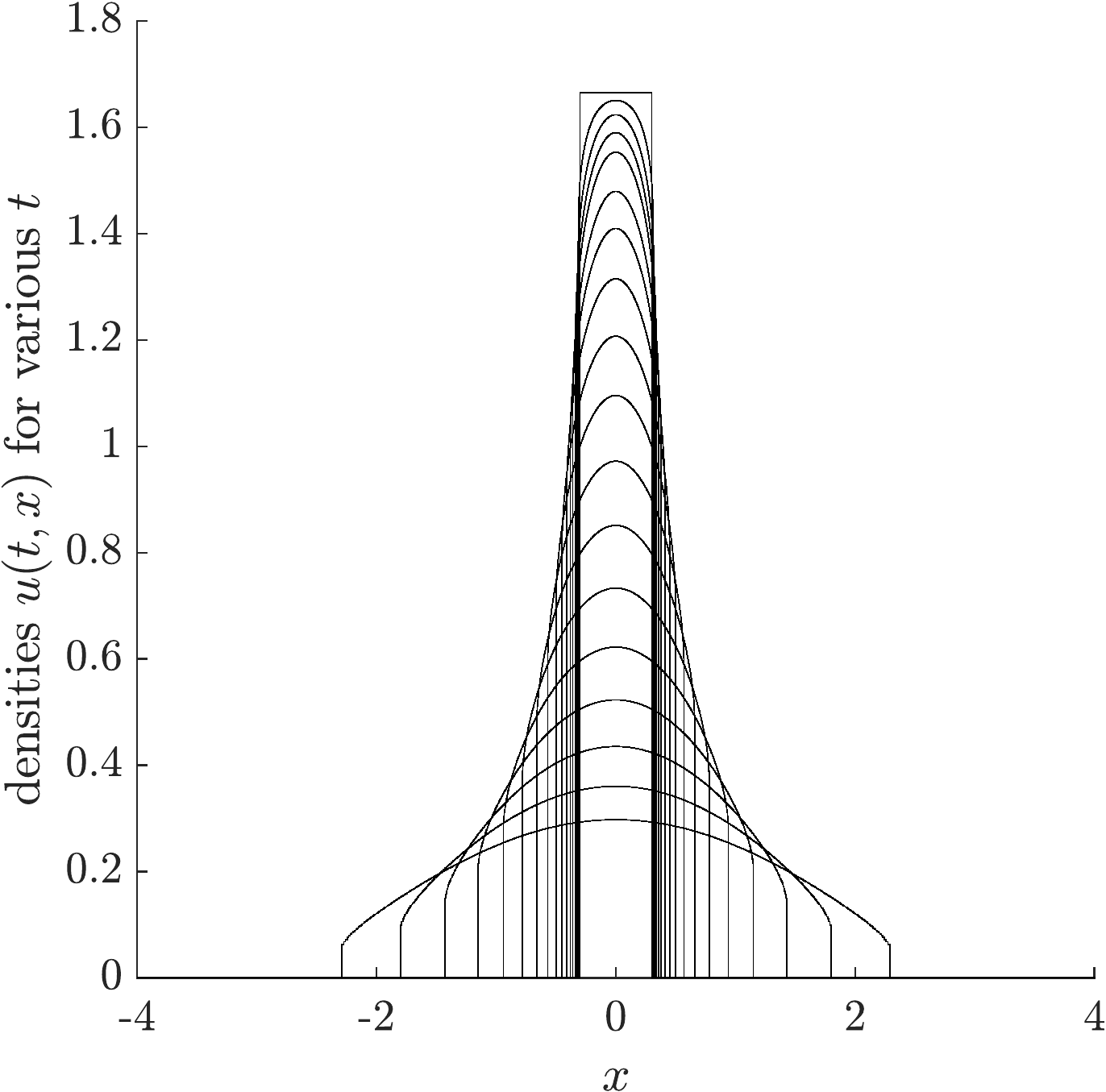}
\includegraphics[width=0.42\textwidth]{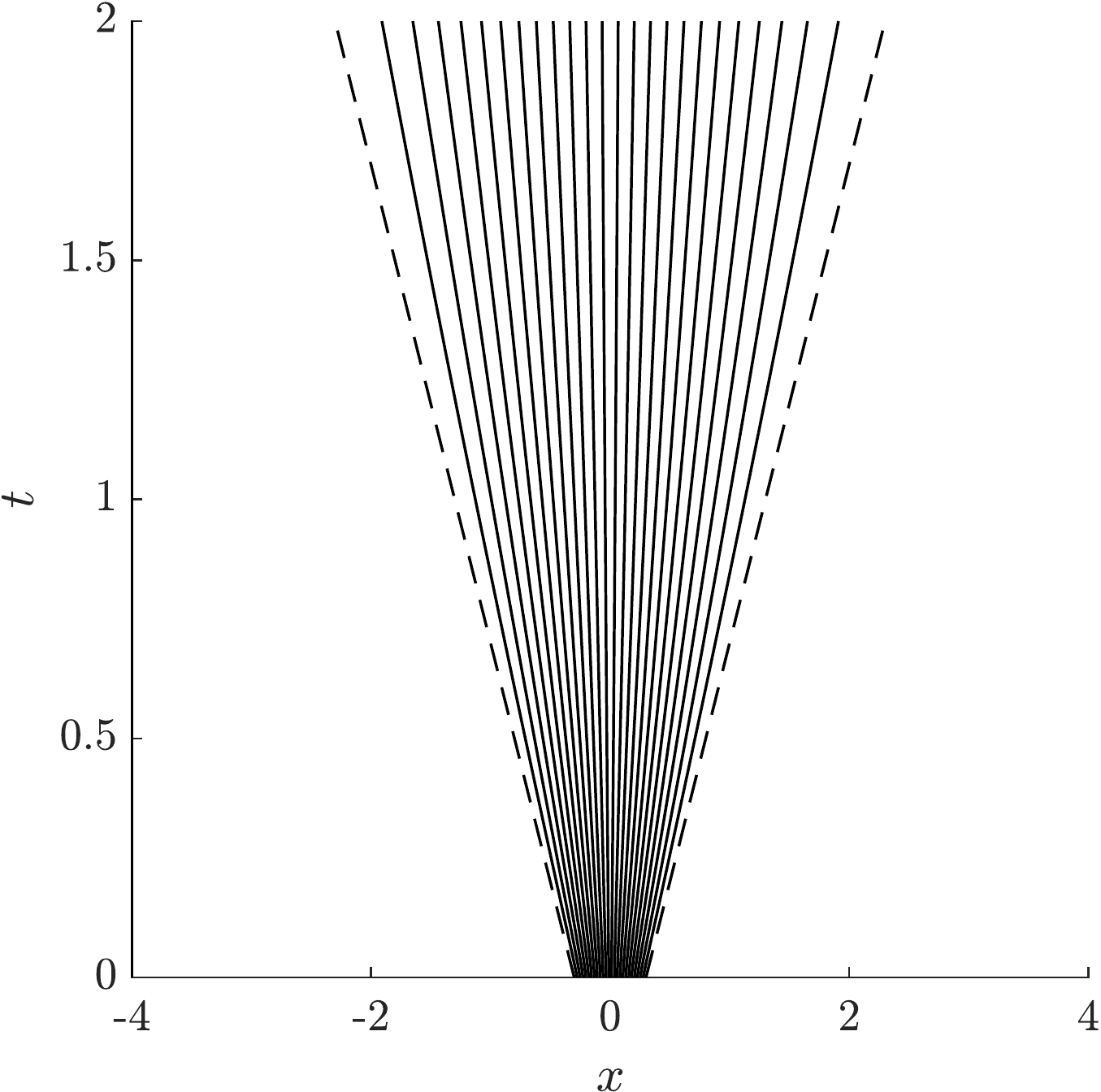}
\captionof{figure}{Experiment: relativistic cost, linear diffusion. Left: Approximate densities $u(t,\cdot)$ for $t=0.01\cdot 10^k$, $k=0,0.12,0.24,\ldots,\log_{10}(200)$, initial mass uniformly distributed on $[-0.3,0.3]$. Right: the corresponding characteristics (dashed: speed of light).}
\label{fig:evol_rel}
\end{figure}

\begin{figure}[h]
\centering
\includegraphics[width=0.445\textwidth]{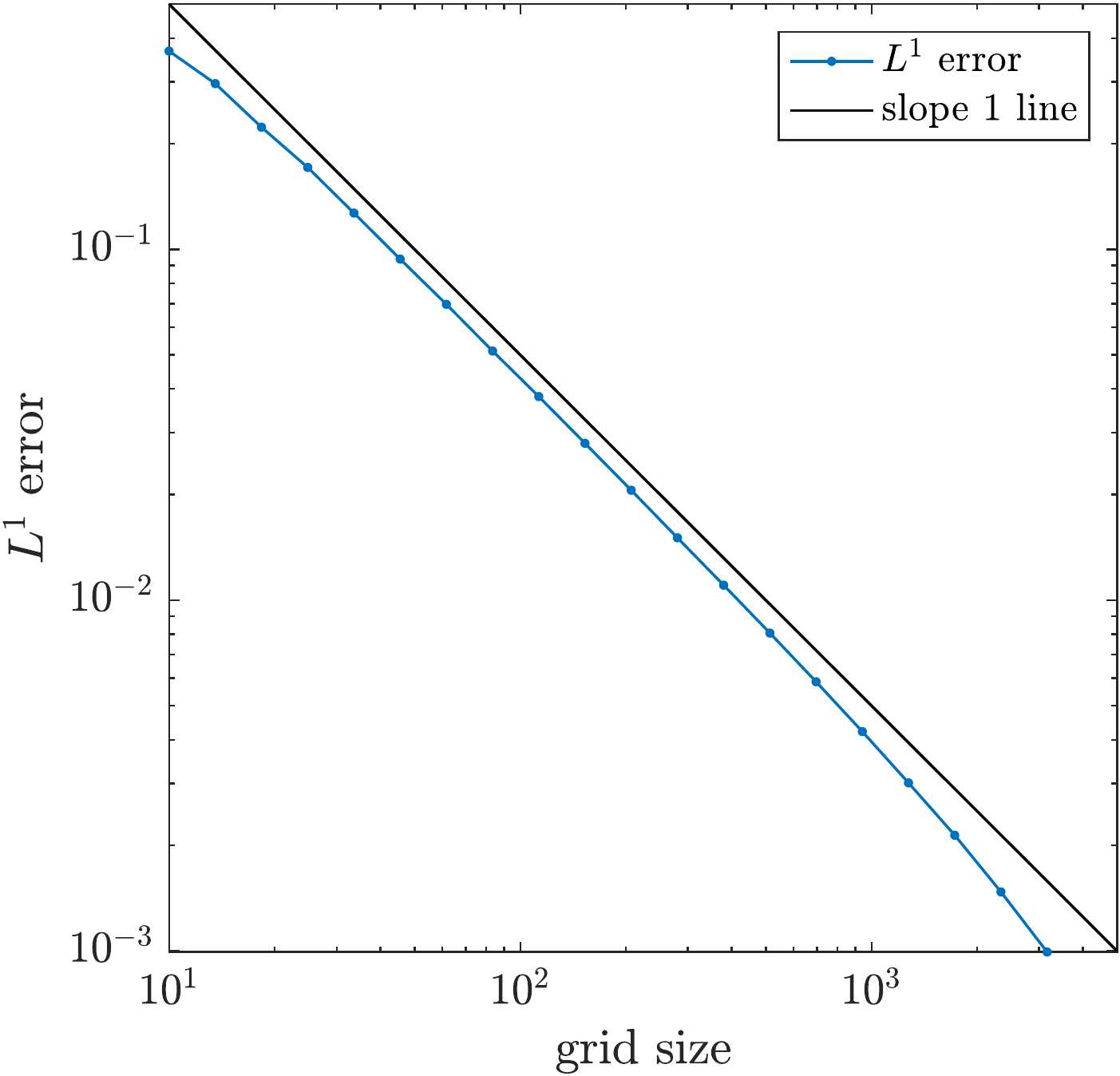}\quad
\includegraphics[width=0.45\textwidth]{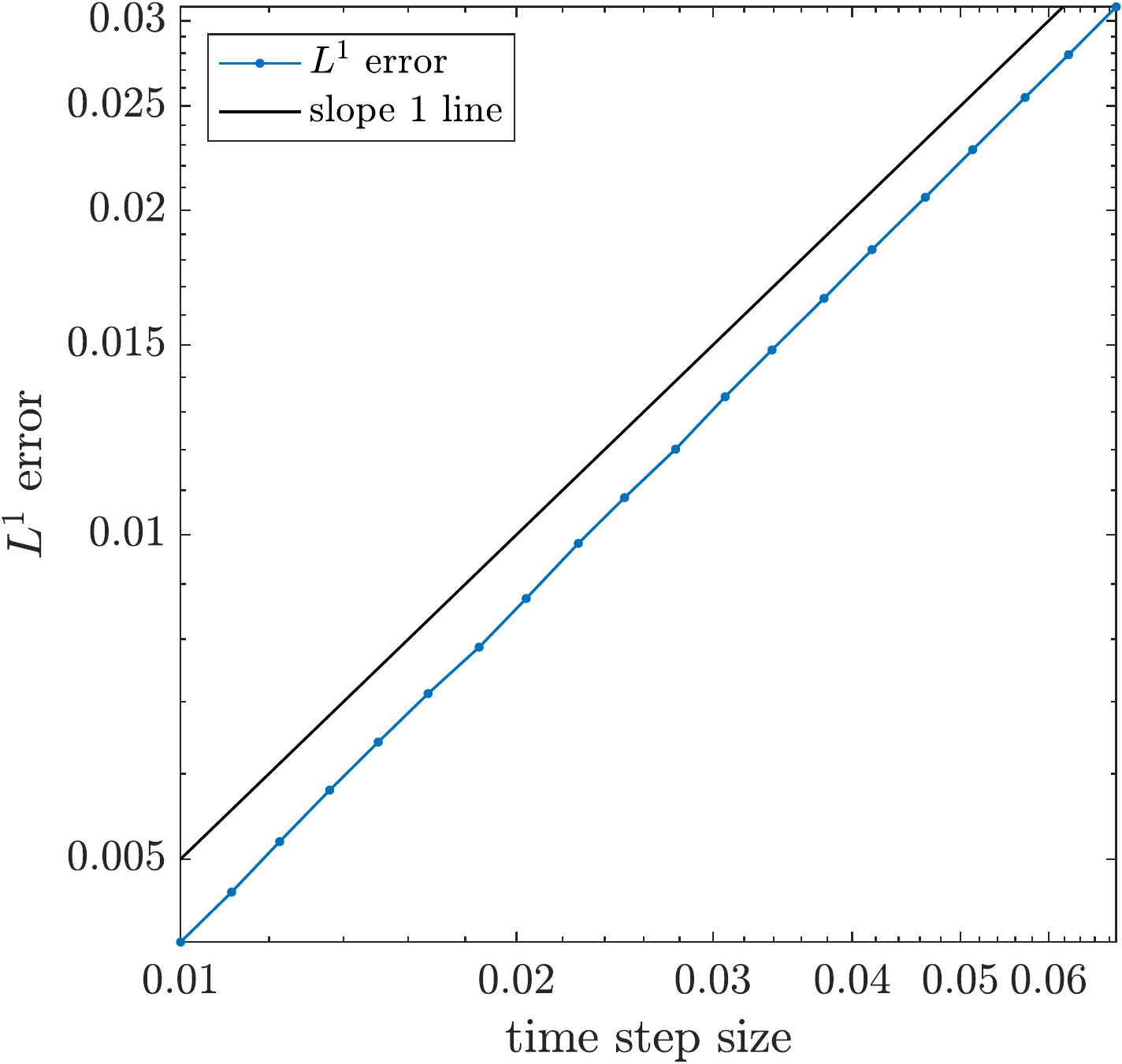}
\caption{Convergence analysis: relativistic cost, linear diffusion. $L^1$-error of the inverse distribution function in dependence of the grid size (left), and in dependence of the time step (right).}
\label{fig:convergence}
\end{figure}

\subsection{Porous medium equation}

As a second experiment, we consider the case of nonlinear diffusion with $m=\frac53$. We choose $p=\frac43$ so that we obtain the $q$-Laplace equation with $q=\frac{2-p}{p-1}=2$.  Fig.~\ref{fig:evol_rel_q-Laplace} shows the evolution of the densities and the associated characteristics.

\begin{figure}[h]
\centering
\includegraphics[width=0.42\textwidth]{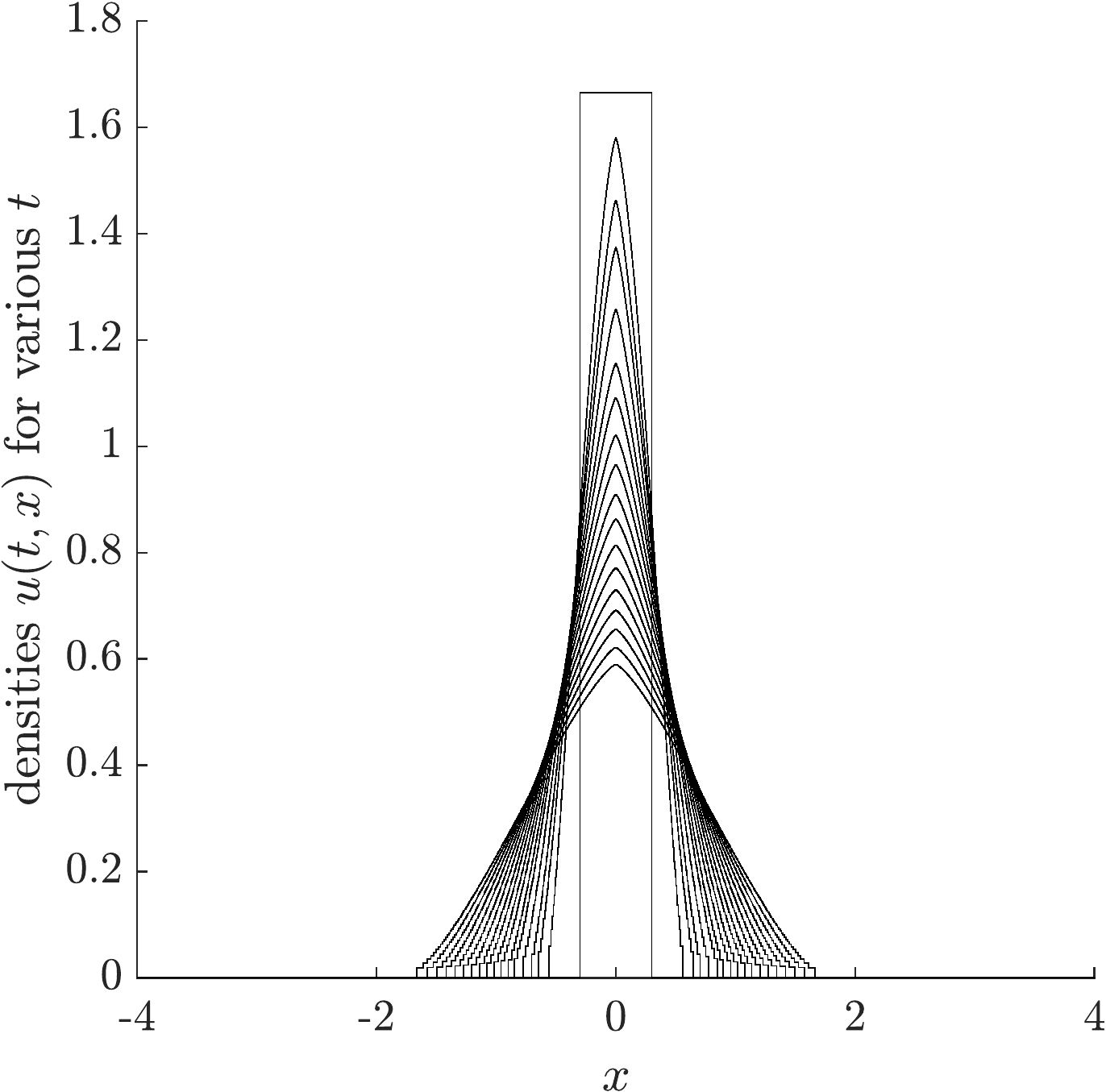}
\includegraphics[width=0.42\textwidth]{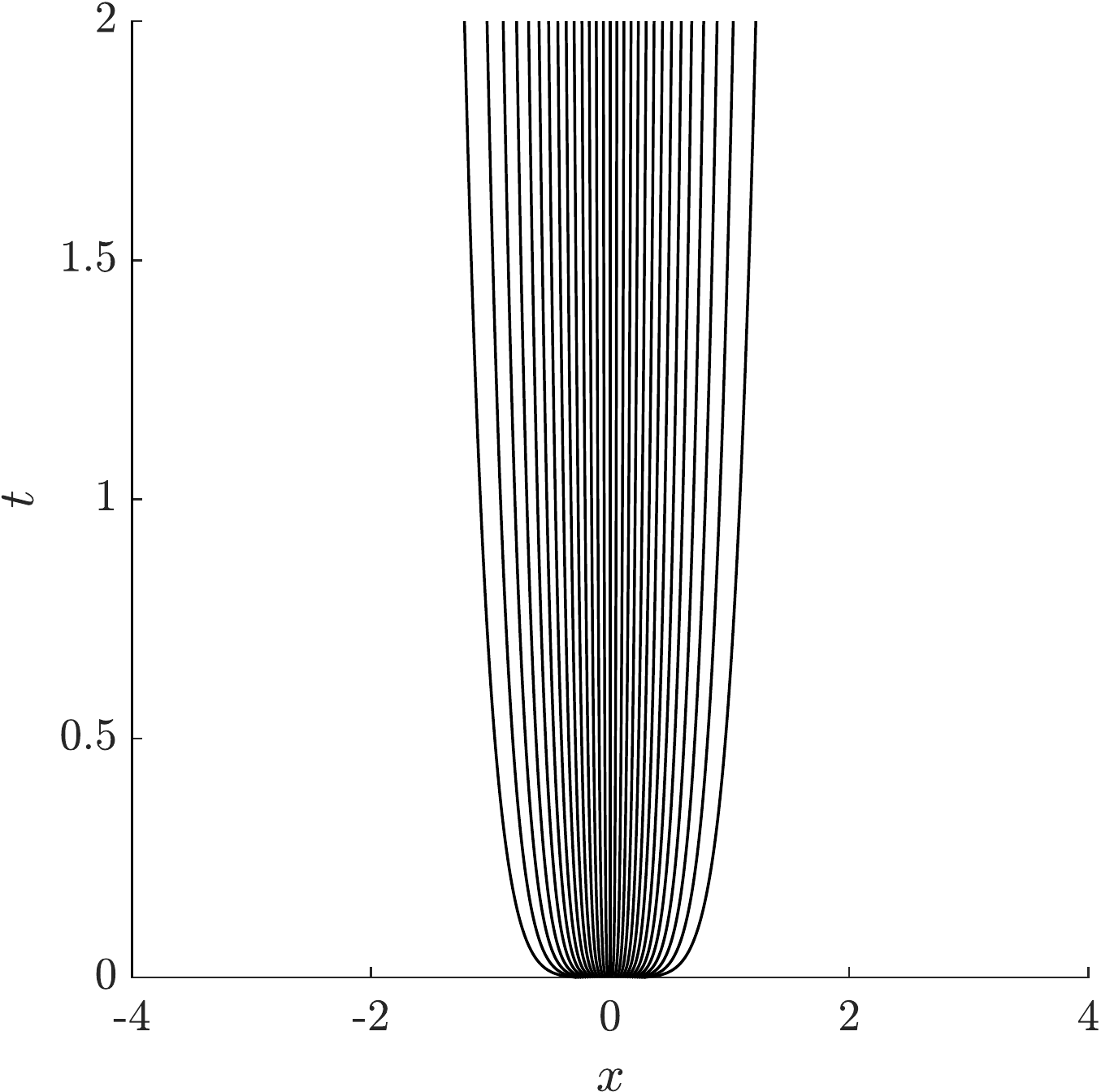}
\captionof{figure}{Experiment: $q$-Laplace ($p=\frac43, m=\frac53$). Left: Approximate densities $u(t,\cdot)$ for $t=0.01\cdot 10^k$, $k=0,0.12,0.24,\ldots,\log_{10}(200)$, initial mass uniformly distributed on $[-0.3,0.3]$. Right: the corresponding characteristics.}
\label{fig:evol_rel_q-Laplace}
\end{figure}

\clearpage

\appendix

\section{Auxiliary convergence results}
\subsection{Difference quotients and weak derivatives}\label{app:diff_quot} (c.f. \cite[Chapter 5.8.2]{Evans}) \\
\emph{
Let $f_k:\Omega\rightarrow \R$ be a sequence of real valued functions on a open rectangle $\Omega\in \R^2$, $r_k\searrow 0$ a sequence and let $p\in (1,\infty)$ as well as $\mathsf e_1, \mathsf{e}_2$ be the canonical basis vectors. If $f_k$ is uniformly bounded w.r.t. $L^p$ and, for some $i=1,2$, $\delta_i f_k(x):=\frac{f_k(x+r_k\mathsf{e}_i)-f_k(x)}{r_k}\ind{\Omega_{r_k,i}}(x)$ with $\Omega_{\varepsilon,i}:=\{x\in \Omega \mid x\pm \varepsilon \mathsf{e}_i\in \Omega\}$ is uniformly bounded w.r.t. $L^p(\Omega)$, then $f_k\rightarrow f_*$ in $L^p(I)$ and $\delta_i f_k\xrightarrow{weakly} \partial_{x_i} f_*$ on $\Omega$. 
}
\begin{proof}
The uniform bounds of $f_k$ and $\delta_i f_k$ in $L^p$ imply weak convergences of an (unrelabeled) subsequence of $f_k$, $\delta_i f_k$ to some limits $f_*:\Omega \rightarrow \R$ and $g:\Omega\rightarrow \R$ respectively, in $L^p$. 

Furthermore, we receive for every $\varepsilon>0$, $\phi\in C^\infty_c(\Omega_{\varepsilon,i})$ and $k$ big enough such that $r_k<\varepsilon$
\begin{align*}
\int_{\Omega_{\varepsilon,i}} g(x) \phi(x) \de x &= \lim_{k\rightarrow \infty} \int_{\Omega_{\varepsilon,i}} \frac{f_k(x+r_k\mathsf{e}_i)-f_k(x)}{r_k} \phi(x) \de x \\
&= - \lim_{k\rightarrow \infty} \int_{\Omega_{\varepsilon,i}} f_k(x) \frac{\phi(x-r_k\mathsf{e}_i)-\phi(x)}{r_k}  \de x \\
&=- \int_{\Omega_{\varepsilon,i}} f_k(x) \partial_{x_i} \phi(x)  \de x \;.
\end{align*}
This argument holds for every $\varepsilon>0$ so the claim follows by uniqueness of the weak derivative.
\end{proof}

\subsection{Strong convergence and Lipschitz-functions}\label{app:str_conv}
\emph{
Let $f_k:\Omega\rightarrow W$ be a sequence of real valued functions on an open $\Omega \in \R^n$ with values in some closed $W$ converging strongly w.r.t. $L^p(\Omega)$ for some $p\in (1,\infty)$ to $f_*:\Omega\rightarrow \R$. Let furthermore $g:\R\rightarrow \R$ be Lipschitz on $W$. Then $g(f_k)\rightarrow g(f_*)$ in $L^p(\Omega)$. 
}

\begin{proof}
Strong convergence of $f_k$ implies uniform boundedness for every subsequence a subsubsequence converging pointwise to some limit. By continuity of $g$ this carries over to $g(f_k)$ and said subsubsequences thereof. On the other hand, $f_k$ being bounded in $L^p$ implies, by Lipschitz-continuity $g(f_k)$ being bounded in $L^p$ and therefore, by the dominated convergence theorem, the pointwise convergence of said subsubsequences of $g(f_k)$ is actually a convergence w.r.t. $L^p$. Finally, since $f_k$ converges to $f_*$. all the cluster points of $g(f_k)$ are $g(f_*)$, which implies the claim. 
\end{proof}

\section{Matlab code}

\begin{lstlisting}[language=Matlab]
a = -4; b = 4;        % domain
k = 10000;            % no of grid points
delta = 1/k;          % mesh size (== mass on each subinterval)
T = 0.7; tau = 0.01;  % time horizon, step size
p = 7;                % cost function: p in (1,Inf) or p=Inf
l = 1.0;              % speed of light
m = 1;                % int. energy, m=1: lin. diff., m \in (1,inf):porous media

c = @(x) (abs(x) > l).*realmax + (abs(x) <= l).*(1-sqrt(1-(x/l).^2)); % for p==Inf
%c = @(x) (1/p).*abs(x).^p;                         % for p < Inf
h = @(x) -log(x);                                   % for m==1
%h = @(x) 1./(m*x.^(m-1));                          % for m > 1

int = @(x) delta*sum(x);                            % integral
d = @(x) (x(2:end)-x(1:end-1))/delta;               % finite difference
F1 = @(x,y) tau*int(c((x-y)/tau)) + int(h(d(x)));   % functional
F = @(x,y) F1([a,x,b],[a,y,b]);                     % fixed boundary
g = @(x) min(diff(x));                              % 

steps = floor(T/tau);                               % number of time steps 
X = zeros(steps,k+1); 
xi = linspace(delta,1.0-delta,k-1);                 % initial condition
r = -3.4; s = 3.7;
X(1,:) = [a, (r-a)/(1-delta)*xi+a+s, b];   

for n = 1:steps-1                                   % time stepping                              
    X(n+1,:) = [a, newton(F,g,X(n,2:end-1)), b]; n
end

figure(1); clf; plot(X,'k');                        % plotting
figure(2); clf; hold on   
for n = 1:steps
    stairs(X(n,:),[delta./diff(X(n,:)),0],'k');
end

function x = newton(F,g,x0)
    x = x0; dx = Inf;
    while norm(dx) > 1e-3
       Fx = F(hessianinit(x),x0);
       dx = (-Fx.hx\Fx.dx')';
       h = 1.0;
       while (F(x+h*dx,x0)>1e300) | (g(x+h*dx)<0), h = h/2; end
       x = x + h*dx;
    end
end
\end{lstlisting}

\bibliography{ConvLagDisc}

\bibliographystyle{plain}

\end{document}